\documentclass[12pt,thmsa]{article}
\usepackage{amssymb}
\usepackage{amsfonts}
\usepackage{enumerate}
\usepackage{amsmath}
\usepackage{cite}
\usepackage[top=3.8cm,bottom=3.8cm,textwidth=16cm,centering,a4paper]{geometry}
\usepackage{tikz}
\usepackage{caption2}
\usepackage{subfigure}
\usepackage{tcolorbox}
\usepackage{hyperref}
\usepackage{cleveref}
\usepackage{indentfirst}

\allowdisplaybreaks

\newtheorem{theorem}{Theorem}[section]

\newtheorem{proposition}[theorem]{Proposition}
\newtheorem{corollary}[theorem]{Corollary}
\newtheorem{remark}[theorem]{Remark}

\newtheorem{lemma}[theorem]{Lemma}

\newenvironment{proof}[1][Proof]{\noindent\textbf{#1.} }{\ \rule{0.5em}{0.5em}}

\begin{document}

\title{On planar Schr\"{o}dinger-Poisson systems with repulsive interactions in the mass supercritical
regime}
\date{}
\author{Juntao Sun$^{a}$\thanks{%
E-mail address: jtsun@sdut.edu.cn(J. Sun)}, Shuai Yao$^{a}$\thanks{%
E-mail address: shyao@sdut.edu.cn (S. Yao)}, He Zhang$^{b}$\thanks{
E-mail address: hehzhang@163.com(H. Zhang)} \\
{\footnotesize $^{a}$\emph{School of Mathematics and Statistics, Shandong
University of Technology, Zibo 255049, PR China }}\\
{\footnotesize $^b$\emph{School of Mathematics and Statistics, Central South
University, Changsha 410083, PR China }}}
\maketitle

\begin{abstract}
In this paper, we investigate solutions with prescribed $L^{2}$-norm (i.e.,
prescribed mass) for the planar Schr\"{o}dinger-Poisson (SP) equation%
\begin{equation*}
-\Delta u+\lambda u+\alpha \left( \log |\cdot |\ast |u|^{2}\right)
u=|u|^{p-2}u,\ \text{in}\ \Omega_{R} ,
\end{equation*}%
where $\lambda \in \mathbb{R}$ is unknown, $\alpha <0,p>4$ and $\Omega_{R}
\subseteq \mathbb{R}^{2}$ is a domain. First, we prove that the energy
functional $J$ corresponding to the SP equation in $\mathbb{R}^{2}$ is
unbounded both above and below on the Pohozaev manifold $\mathcal{P}$; this
explains the reason why the minimax level of $J$ is difficult to determine,
as referenced in [Cingolani and Jeanjean, SIAM J. Math. Anal., 2019].
Second, we establish the existence of a ground state and a high-energy
solution, both with positive energy in a large bounded domain $\Omega_{R} $,
which is a substantial advancement in addressing an open problem proposed in
[Cingolani and Jeanjean, SIAM J. Math. Anal., 2019]. Finally, we analyze the
asymptotic behavior of solutions as the domain $\Omega_{R} $ is extended to
the entire space $\mathbb{R}^{2}$.
\end{abstract}



\section{Introduction}

Our starting point is the planar Schr\"{o}dinger--Poisson system of the form
\begin{equation}
\left\{
\begin{array}{ll}
i\partial _{t}\phi -\Delta \phi +\alpha w\phi =|\phi |^{p-2}\phi ,\quad &
\forall (t,x)\in \mathbb{R}^{1+2}, \\
\Delta w=|\phi |^{2}, &
\end{array}%
\right.  \label{E1}
\end{equation}%
where $\phi :\mathbb{R}\times \mathbb{R}^{2}\rightarrow \mathbb{C}$ is the
time-dependent wave function, $w$ represents the Newtonian potential
corresponding to the nonlocal self-interaction of $\phi $, the coupling
constant $\alpha \in \mathbb{\mathbb{R}}$ characterizes the relative
strength of the potential, and the sign of $\alpha $ determines whether the
potential interaction is attractive or repulsive: specifically, the
interaction is attractive if $\alpha >0$ or repulsive if $\alpha <0$. Such a
system originates from quantum mechanics \cite{BBL,CL,L} and semiconductor
theory \cite{Li,MRS}.

An interesting topic related to System (\ref{E1}) is to find standing waves
of the form $\phi (x,t)=e^{-i\lambda t}u(x),$ where $\lambda \in \mathbb{R}$
and $u:\mathbb{R}^{2}\rightarrow \mathbb{R}.$ Substituting this into System (%
\ref{E1}) reduces it to the following system:%
\begin{equation}
\left\{
\begin{array}{ll}
-\Delta u+\lambda u+\alpha wu=|u|^{p-2}u, & \text{ in }\mathbb{R}^{2}, \\
\Delta w=u^{2}, & \text{ in }\mathbb{R}^{2}.%
\end{array}%
\right.  \label{E2}
\end{equation}%
The second equation of System (\ref{E2}) only determines $w$ up to harmonic
functions. It is natural to choose $w$ as the negative Newton potential of $%
u^{2}$, i.e., the convolution of $u^{2}$ with the fundamental solution $\Phi
$ of the Laplacian, where $\Phi (x)=\frac{1}{2\pi }\log |x|.$ Through this
formal inversion, System (\ref{E2}) is converted into an
integro-differential equation:%
\begin{equation}
\begin{array}{ll}
-\Delta u+\lambda u+\alpha (\Phi \ast |u|^{2})u=|u|^{p-2}u, & \ \text{in}\
\mathbb{R}^{2}.%
\end{array}
\label{E3}
\end{equation}%
A key feature of Equation (\ref{E3}) is the presence of a logarithmic
convolution potential, which is unbounded and sign-changing. For Equation (%
\ref{E3}), the applicability of variational methods non-trivial, as the
associated energy functional is not well-defined on $H^{1}(\mathbb{R}^{2})$
due to the influence of the logarithmic convolution potential. Inspired by
Stubbe \cite{S2008}, Cingolani and Weth \cite{CW2016} developed a
variational framework for Equation (\ref{E3}) in the smaller Hilbert space $X
$, defined as:%
\begin{equation*}
X:=\left\{ u\in H^{1}(\mathbb{R}^{2})\text{ }|\ \int_{\mathbb{R}^{2}}\log
\left( 1+|x|\right) |u|^{2}dx<\infty \right\} ,
\end{equation*}%
and proved the existence of ground state solutions when $p>4.$ Since then,
the study of nontrivial solutions for Equation (\ref{E3}) with a fixed
positive frequency $\lambda >0$ has garnered significant attention in recent
years. For related results involving different types of nonlinearities, we
refer the reader to \cite{ACFM,A2021,CT2021,CT2020,DW2017,LRZ2023,PWZ2024}.

When the frequency $\lambda $ is unknown in Equation (\ref{E3}), the $L^{2} $%
-norm of solutions to Equation (\ref{E3}) are prescribed, i.e., $\int_{%
\mathbb{R}^{2}}|u|^{2}dx=\rho $ for a given $\rho >0$; such solutions are
usually referred to as normalized solutions. From a physical perspective,
this study is particularly meaningful because solutions to System (\ref{E1})
conserve their mass over time. Additionally, from a mathematical standpoint,
investigating this scenario is more challenging due to the additional
requirement of satisfying the $L^{2}$-constraint. To the best of our
knowledge, the first contribution along this direction was recently made by
Cingolani and Jeanjean \cite{CJ2019}, where they established the existence
of solutions for the following problem
\begin{equation}
\left\{
\begin{array}{ll}
-\Delta u+\lambda u+\alpha \left( \log |\cdot |\ast |u|^{2}\right) u=\beta
|u|^{p-2}u, & \ \text{in}\ \mathbb{R}^{2}, \\
\int_{\mathbb{R}^{2}}|u|^{2}dx=\rho , &
\end{array}%
\right.  \label{E4}
\end{equation}%
depending on the parameters $\alpha ,\beta \in \mathbb{R}$ and $p>2.$ From
the variational perspective, solutions to Problem (\ref{E4}) can be obtained
as critical points of the energy functional $J(u):X\rightarrow \mathbb{R}$
defined by
\begin{equation*}
J(u)=\frac{1}{2}\int_{\mathbb{R}^{2}}|\nabla u|^{2}dx+\frac{\alpha }{4}\int_{%
\mathbb{R}^{2}}\int_{\mathbb{R}^{2}}\log \left( |x-y|\right)
|u(x)|^{2}|u(y)|^{2}dxdy-\frac{\beta }{p}\int_{\mathbb{R}^{2}}|u|^{p}dx
\end{equation*}%
under the constraint
\begin{equation*}
\mathcal{S}_{\rho }:=\left\{ u\in X:\int_{\mathbb{R}^{2}}|u|^{2}dx=\rho
\right\} .
\end{equation*}%
More precisely, it has been proven that $(i)$ a ground state solution exists
as a global minimizer of $J$ on $\mathcal{S}_{\rho }$ if $\alpha >0,\beta
\leq 0,$ and $p>2;$ or $\alpha >0,\beta >0,$ and $p<4;$ or $\alpha >0,\beta
>0,$ and $p=4;$ $(ii)$ two solutions exist if $\alpha >0,\beta >0,$ and $%
p>4; $ or $\alpha <0,0<K_{1}(\alpha )<\beta <K_{2}(\alpha ),$ and $p<4,$ by
virtue of the decomposition of Pohozaev manifold
\begin{equation*}
\mathcal{P}=\left\{ u\in \mathcal{S}_{\rho }:P(u)=0\right\} ,
\end{equation*}%
where $P(u)=0$ denotes the Pohozaev identity associated with Problem (\ref%
{E4}), as developed in \cite{So2020,So2020-1}; $(ii)$ no solution exists
when $\alpha <0,\beta \leq 0,$ and $p>2.$ For the case where the
nonlinearity in Problem (\ref{E4}) exhibits exponential critical growth,
several existence results have also been obtained in \cite{ABM,CRT,SYZ2024}.

In \cite{CJ2019}, Cingolani and Jeanjean pointed out that the case where $%
\alpha <0,\beta >0,$ and $p>4$ remains completely open, and it is unclear
how to identify a minimax level of the energy functional $J$. Motivated by
this observation, in this paper we delve into the underlying reasons for the
difficulty in determining such a minimax level and attempt to make
substantial progress on this open problem.

Despite the apparent simplicity of the geometric structure of $J$ when $%
\alpha <0,\beta >0,$ and $p>4,$ in particular for any $u\in S_{\rho }$ there
exists a unique $t>0$ such that $tu(tx)=:u_{t}\in \mathcal{P}$, we argue
that it is actually highly worthy of in-depth investigation. This
constitutes the first main result. Set
\begin{equation*}
M_{\rho }:=\inf_{u\in \mathcal{P}}J(u)\text{ and }\overline{M}_{\rho
}:=\sup_{u\in \mathcal{P}}J(u).
\end{equation*}%
We establish the following result.

\begin{theorem}
\label{T1.1} Let $\alpha <0,\beta >0,$ and $p>4$. Then $M_{\rho }=-\infty $
and $\overline{M}_{\rho }=+\infty $.
\end{theorem}

\begin{remark}
\label{R1-0} By Theorem \ref{T1.1} we find that the energy functional $J$ is
unbounded both above and below on $\mathcal{P}$. This reveals the reason why
the minimax level of $J$ is difficult to determine. Therefore, the commonly
used methods, such as the Pohozaev manifold and the Mountain Pass Theorem,
are no longer applicable to this case.
\end{remark}

Owing to the intractability of the open problem, we restrict our
investigation to a large bounded domain, namely%
\begin{equation}
\left\{
\begin{array}{ll}
-\Delta u+\lambda u+\alpha \left( \log |\cdot |\ast |u|^{2}\right)
u=|u|^{p-2}u, & \ \text{in}\ \Omega _{R}, \\
\int_{\Omega _{R}}|u|^{2}dx=\rho , &
\end{array}%
\right.  \label{E5}
\end{equation}%
where $\alpha <0,p>4,$ and $\Omega _{R}:=\{Rx\in \mathbb{R}^{N}:x\in \Omega
\}$ with $R>0.$ Here $\Omega $ is a bounded domain satisfying $\max_{x,y\in
\Omega }|x-y|=1$. Our second objective is to find normalized solutions for
Problem (\ref{E5}).

Note that compared to the study of normalized solutions in the entire space,
there appear to be only a few results focusing on normalized solutions in
bounded domains. In fact, these two settings are fundamentally distinct:
each requires a specific approach, and the results are generally not
comparable. For instance, the Pohozaev manifold is not applicable in bounded
domains, as the invariance under translation and dilation is lost.

Before stating main results, we introduce a few notations. Let $\Vert \cdot
\Vert _{q}$ and $\Vert \cdot \Vert _{q,\Omega }$ denote the norms of the
spaces $L^{q}(\mathbb{R}^{N})$ and $L^{q}(\Omega ),$ respectively. Let $%
\mathcal{C}_{p}$ be the best constant in the Gagliardo-Nirenberg inequality (%
\cite[Theorem in Lecture II]{N1959}):
\begin{equation}
\Vert u\Vert _{q}^{q}\leq \mathcal{C}_{q}\Vert u\Vert _{2}^{2}\Vert \nabla
u\Vert _{2}^{q-2}\ \text{for }q>2.  \label{E6}
\end{equation}%
Let $\widetilde{J}_{R}(u)=J(u)|_{\Omega _{R}}$ with $\beta =1$ and $\mathcal{%
S}_{\Omega _{R},\rho }=\mathcal{S}_{\rho }|_{\Omega _{R}}$. Then we
establish the following results.

\begin{theorem}
\label{T1.3} Let $\alpha <0,p>4,$ and $\rho >0$. Then there exists $R_{\rho
}>0$ such that for any $R>R_{\rho },$ the following statements hold true.

\begin{itemize}
\item[$(i)$] There exists $\alpha _{R,\rho }^{\ast }=\alpha (R,\rho )>0$
such that for any $0<|\alpha |<\alpha _{R,\rho }^{\ast },$ Problem (\ref{E5}%
) admits a local minimum type solution $u_{R}^{0}\in H_{0}^{1}(\Omega _{R})$
satisfying $\widetilde{J}_{R}(u_{R}^{0})>0$ for some $\lambda =\lambda
_{R}^{0}\in \mathbb{R}$.

\item[$(ii)$] If, in addition, $\Omega $ is a star-shaped domain, then there
exists $R_{\ast }\geq R_{\rho }$ such that for any $R>R_{\ast }$ and $%
0<|\alpha |<\alpha _{R,\rho }^{\ast },$ the solution $u_{R}^{0}$ obtained in
$(i)$ is a ground state.

\item[$(iii)$] If, in addition, $\Omega $ is a star-shaped domain, then for
any $0<|\alpha |<\alpha _{R,\rho }^{\ast },$ Problem (\ref{E5}) admits a
second solution of mountain-pass type $u_{R}^{1}\in H_{0}^{1}(\Omega _{R})$
satisfying $\widetilde{J}_{R}(u_{R}^{1})>\widetilde{J}_{R}(u_{R}^{0})$ for
some $\lambda =\lambda _{R}^{1}\in \mathbb{R}$.
\end{itemize}
\end{theorem}

In Theorem \ref{T1.3}, if we assume $R$ is fixed (e.g. $R=1$), then the
existence result still holds.

\begin{corollary}
\label{T1.4} Let $\alpha <0$ and $p>4.$ Then there exists $0<\rho
_{0}\leq\rho^{\ast}$  such that for any $0<\rho \leq \rho _{0},$ there
exists $\alpha _{\rho }^{\ast }=\alpha (\rho )>0$ such that for any $%
0<|\alpha |<\alpha _{\rho }^{\ast },$ Problem (\ref{E5}) with $R=1$ admits a
local minimizer $u^{0}\in H_{0}^{1}(\Omega )$ for some $\lambda =\lambda
^{0}\in \mathbb{R}$ and a second solution of mountain-pass type $u^{1}\in
H_{0}^{1}(\Omega )$ for some $\lambda =\lambda ^{1}\in \mathbb{R}$.
Moreover, there exists $0<\rho _{1}\leq \rho _{0}$ such that the local
minimizer $u^{0}$ is a ground state for any $0<\rho \leq
\rho _{1}$ and $0<|\alpha |<\alpha _{\rho }^{\ast }$. Here%
\begin{equation*}
\rho ^{\ast }:=(\lambda _{1}(\Omega ))^{-\frac{p-4}{p-2}}\left[ \frac{p}{%
(p-2)\mathcal{C}_{p}}\right] ^{\frac{2}{p-2}},
\end{equation*}%
where $\lambda _{1}(\Omega )$ is the principal eigenvalue of $-\Delta $ with
Dirichlet boundary condition in $\Omega $, and $\mathcal{C}_{p}$ is given in
(\ref{E6}).
\end{corollary}

\begin{theorem}
\label{T1.5} Let $\left( \lambda _{R}^{i},u_{R}^{i}\right) \in \mathbb{R}%
\times H_{0}^{1}(\Omega _{R})$ ($i=0,1$) be solutions to Problem (\ref{E5})
given by Theorem \ref{T1.3}. Then%
\begin{equation*}
\left\Vert \nabla u_{R}^{0}\right\Vert _{2,\Omega _{R}}\rightarrow 0,\lambda
_{R}^{1}\rightarrow \widetilde{\lambda }^{1}\text{ in }\mathbb{R}\text{ and }\ u_{R}^{1}\rightarrow \widetilde{u}^{1}\text{ in }H^{1}(
\mathbb{R}^{2})\text{ as }R\rightarrow \infty ,
\end{equation*}%
where $(\widetilde{\lambda }^{1},\widetilde{u}^{1})$ is the unique solution
to the following problem
\begin{equation*}
\left\{
\begin{array}{ll}
-\Delta u+\lambda u=|u|^{p-2}u, & \ \text{in}\ \mathbb{R}^{2}, \\
\int_{\mathbb{R}^{2}}|u|^{2}dx=\rho , &
\end{array}%
\right.
\end{equation*}%
with $\widetilde{u}^{1}>0$.
\end{theorem}

\begin{remark}
\label{R1-1} $(i)$ By Theorem \ref{T1.3}, Problem (\ref{E5}) admits two
normalized solutions with positive energy on a large bounded domain $\Omega
_{R}$. However, Theorem \ref{T1.5} reveals that extending the domain $\Omega
_{R}$ to the entire space $\mathbb{R}^{2}$ is not a flexible approach to
obtain solutions of Problem (\ref{E4}), which differs from the study in Schr%
\"{o}dinger equations \cite{BQZ}. The primary reason is that the logarithmic
kernel is sign-changing.\newline
$(ii)$ It is noted that if we adopt the proof of Theorem \ref{T1.3} $(i)$ to
establish the existence of local minimizers for Problem (\ref{E4}), it
becomes impossible to appropriately estimate the nonlocal term, thereby
preventing us from finding a suitable ball radius.
\end{remark}

The rest of this paper is organized as follows. Theorem \ref{T1.1} is proven
in Section 2. Theorems \ref{T1.3} and \ref{T1.5} are established in Section
3.

\section{The property of the energy functional $J$ on $\mathcal{P}$}

As presented in \cite{CW2016}, we define the symmetric bilinear forms
\begin{eqnarray*}
&&(u,v)\mapsto \chi _{1}(u,v)=\int_{\mathbb{R}^{2}}\int_{\mathbb{R}^{2}}\log
\left( 1+|x-y|\right) u(x)v(y)dxdy, \\
&&(u,v)\mapsto \chi _{2}(u,v)=\int_{\mathbb{R}^{2}}\int_{\mathbb{R}^{2}}\log
\left( 1+\frac{1}{|x-y|}\right) u(x)v(y)dxdy, \\
&&(u,v)\mapsto \chi _{0}(u,v)=\chi _{1}(u,v)-\chi _{2}(u,v)=\int_{\mathbb{R}%
^{2}}\int_{\mathbb{R}^{2}}\log \left( |x-y|\right) u(x)v(y)dxdy,
\end{eqnarray*}%
and functions on $X$ as follows:
\begin{eqnarray*}
&&\chi _{1}(u^{2},v^{2})=\int_{\mathbb{R}^{2}}\int_{\mathbb{R}^{2}}\log
\left( 1+|x-y|\right) |u(x)|^{2}|v(y)|^{2}dxdy \\
&&\chi _{2}(u^{2},v^{2})=\int_{\mathbb{R}^{2}}\int_{\mathbb{R}^{2}}\log
\left( 1+\frac{1}{|x-y|}\right) |u(x)|^{2}|v(y)|^{2}dxdy.
\end{eqnarray*}%
Note that $\chi _{0}(u^{2},v^{2})=\chi _{1}(u^{2},v^{2})-\chi
_{2}(u^{2},v^{2})$ and it follows from the Hardy-Littlewood Sobolev
inequality \cite{L01} that there exists a constant $\mathcal{C}>0$ such that
\begin{equation}
|\chi _{2}(u^{2},u^{2})|\leq \mathcal{C}\left\Vert u\right\Vert _{8/3}^{4}
\label{E7}
\end{equation}%
for all $u\in L^{8/3}(\mathbb{R}^{2}).$

For any $u\in \mathcal{S}_{\rho }$ and $t>0$, we define $h_{u}(t):\mathbb{R}%
^{+}\rightarrow \mathbb{R}$ by%
\begin{eqnarray*}
h_{u}(t)=J(u_{t}) &=&\frac{1}{2}t^{2}\int_{\mathbb{R}^{2}}|\nabla u|^{2}dx+%
\frac{\alpha }{4}\int_{\mathbb{R}^{2}}\log \left( |x-y|\right)
|u(x)|^{2}|u(y)|^{2}dxdy \\
&&-\frac{\alpha \rho ^{2}}{4}\log t-\frac{\beta }{p}t^{p-2}\int_{\mathbb{R}%
^{2}}|u|^{p}dx.
\end{eqnarray*}%
Then we have the following lemma.

\begin{lemma}
\label{L3.1} Let $\alpha <0,\beta >0,p>4,$ and $\rho >0$. For any $u\in
\mathcal{S}_{\rho }$, there exists a unique $t_{u}>0$ such that $%
u_{t_{u}}\in \mathcal{P}$, and $t_{u}$ is a strict global maximum point of $%
h_{u}$. Furthermore, $h_{u}^{\prime }\left( t_{u}\right) =\frac{P\left(
u_{t_{u}}\right) }{t_{u}}=0$, $h_{u}^{\prime }(t)>0$ for $0<t<t_{u},$ and $%
h_{u}^{\prime }(t)<0$ for $t>t_{u}$, where
\begin{equation*}
P(u):=\Vert \nabla u\Vert _{2}^{2}-\frac{\beta (p-2)}{p}\Vert u\Vert
_{p}^{p}-\frac{\alpha \rho ^{2}}{4}.
\end{equation*}
\end{lemma}

\begin{proof}
A direct calculation shows that
\begin{eqnarray}
h_{u}^{\prime }(t) &=&t\int_{\mathbb{R}^{2}}|\nabla u|^{2}dx-\frac{\alpha
\rho ^{2}}{4t}-\frac{\beta (p-2)}{p}t^{p-3}\int_{\mathbb{R}^{2}}|u|^{p}dx
\label{E8} \\
&=&\frac{1}{t}\left( t^{2}\int_{\mathbb{R}^{2}}|\nabla u|^{2}dx-\frac{\alpha
\rho ^{2}}{4}-\frac{\beta (p-2)}{p}t^{p-2}\int_{\mathbb{R}%
^{2}}|u|^{p}dx\right) .  \notag
\end{eqnarray}%
Obviously, $h_{u}^{\prime }(t)=\frac{P\left( u_{t}\right) }{t}$. By (\ref{E8}%
), there exists $t_{u}>0$ such that $h_{u}^{\prime }(t)$ has a unique zero
point $t_{u}$. Moreover, $h_{u}^{\prime }(t)>0$ for $0<t<t_{u},$ and $%
h_{u}^{\prime }(t)<0$ for $t>t_{u}$. The proof is completed.
\end{proof}

Let $\Omega $ be a bounded domain in $\mathbb{R}^{2}$ satisfying $%
\max_{x,y\in \Omega }|x-y|=1$ and $\mathcal{S}_{\Omega ,\rho }=\mathcal{S}%
_{\rho }|_{\Omega }.$ Then we prove the following result.

\begin{theorem}
\label{L3.2} Let $\alpha <0,\beta >0,p>4$ and $\rho >0$. It holds that $%
M_{\rho }=-\infty $ and $\overline{M}_{\rho }=+\infty .$
\end{theorem}

\begin{proof}
Let $\psi \in \mathcal{S}_{\Omega ,\rho }$ be the positive normalized
eigenfunction $-\Delta $ with Dirichlet boundary condition in $\Omega $
associated to $\lambda _{1}(\Omega )$. Then
\begin{equation}
\Vert \nabla \psi \Vert _{2}^{2}=\lambda _{1}\left( \Omega \right) \rho ,\
\Vert \psi \Vert _{2}^{2}=\rho \ \text{and }\Vert \psi \Vert _{p}^{p}\leq
\mathcal{C}_{p}\left[ \lambda _{1}\left( \Omega \right) \rho \right]
^{(p-2)/2}\rho .  \label{E9}
\end{equation}%
Set $V_{n}:=\sum_{i=1}^{2}2^{\frac{1}{p-4}}\psi \left( 2^{\frac{p-2}{2(p-4)}%
}\left( x+in^{2}e\right) \right) $ for $n\in
\mathbb{N}
$. Then it follows from (\ref{E9}) that
\begin{eqnarray}
&&\int_{\mathbb{R}^{2}}|V_{n}|^{2}dx=\rho ,\ \int_{\mathbb{R}^{2}}|\nabla
V_{n}|^{2}dx=2^{\frac{p-2}{p-4}}\int_{\mathbb{R}^{2}}|\nabla \psi |^{2}dx=2^{%
\frac{p-2}{p-4}}\lambda _{1}\left( \Omega \right) \rho ,  \notag \\
&&\int_{\mathbb{R}^{2}}|V_{n}|^{p}dx=2^{\frac{p-2}{p-4}}\int_{\mathbb{R}%
^{2}}|\psi |^{p}dx\leq 2^{\frac{p-2}{p-4}}\mathcal{C}_{p}\left[ \lambda
_{1}\left( \Omega \right) \rho \right] ^{(p-2)/2}\rho ,  \label{E10}
\end{eqnarray}%
and
\begin{eqnarray}
&&\int_{\mathbb{R}^{2}}\int_{\mathbb{R}^{2}}\log \left( |x-y|\right)
|V_{n}(x)|^{2}|V_{n}(y)|^{2}dxdy  \notag \\
&=&2^{-2}\int_{\mathbb{R}^{2}}\int_{\mathbb{R}^{2}}\log \left( |x-y|\right)
|\sum_{i=1}^{2}\psi \left( x+in^{2}e\right) |^{2}|\sum_{j=1}^{2}\psi \left(
y+jn^{2}e\right) |^{2}dxdy  \notag \\
&&-\frac{p-2}{2(p-4)}\rho ^{2}\log 2  \notag \\
&=&2^{-1}\int_{\mathbb{R}^{2}}\int_{\mathbb{R}^{2}}\log \left( |x-y|\right)
|\psi (x)|^{2}|\psi (y)|^{2}dxdy-\frac{p-2}{2(p-4)}\rho ^{2}\log 2  \notag \\
&&+2^{-2}\sum_{i\neq j}^{2}\int_{\mathbb{R}^{2}}\int_{\mathbb{R}^{2}}\log
\left( |x-y|\right) |\psi (x+in^{2}e)|^{2}|\psi (y+jn^{2}e)|^{2}dxdy  \notag
\\
&=&2^{-1}\int_{\mathbb{R}^{2}}\int_{\mathbb{R}^{2}}\log \left( |x-y|\right)
|\psi (x)|^{2}|\psi (y)|^{2}dxdy-\frac{p-2}{2(p-4)}\rho ^{2}\log 2  \notag \\
&&+2^{-2}\sum_{i\neq j}^{2}\int_{\mathbb{R}^{2}}\int_{\mathbb{R}^{2}}\log
\left( |x-y+(j-i)n^{2}e|\right) |\psi (x)|^{2}|\psi (y)|^{2}dxdy  \notag \\
&\geq &2^{-1}\int_{\mathbb{R}^{2}}\int_{\mathbb{R}^{2}}\log \left(
|x-y|\right) |\psi (x)|^{2}|\psi (y)|^{2}dxdy+\frac{\log \left(
n^{2}-1\right) \rho ^{2}}{2}  \notag \\
&&-\frac{p-2}{2(p-4)}\rho ^{2}\log 2  \notag \\
&\geq &-2^{-1}\mathcal{C}\mathcal{C}_{8/3}^{3/2}(\lambda _{1}(\Omega
))^{1/2}\rho ^{2}+\frac{\log \left( n^{2}-1\right) \rho ^{2}}{2}-\frac{p-2}{%
2(p-4)}\rho ^{2}\log 2.  \label{E11}
\end{eqnarray}%
Using (\ref{E10}), we have
\begin{equation*}
P\left( (V_{n})_{t}\right) =2^{\frac{p-2}{p-4}}\left( t^{2}\int_{\mathbb{R}%
^{2}}|\nabla \psi |^{2}dx-\frac{p-2}{p}t^{p-2}\int_{\mathbb{R}^{2}}|\psi
|^{p}dx-2^{-\frac{p-2}{p-4}}\frac{\alpha \rho ^{2}}{4}\right) \text{ for }%
t>0.
\end{equation*}%
By Lemma \ref{L3.1}, there exists $t_{0}>0$ independent of $n$ such that $%
P\left( (V_{n})_{t_{0}}\right) =0,$ which shows that $(V_{n})_{t_{0}}\in
\mathcal{P}$. Thus it follows from (\ref{E10})--(\ref{E11}) that
\begin{eqnarray*}
&&J\left( (V_{n})_{t_{0}}\right) \\
&=&\frac{p-4}{2(p-2)}\int_{\mathbb{R}^{2}}|\nabla (V_{n})_{t_{0}}|^{2}dx+%
\frac{\alpha }{4}\int_{\mathbb{R}^{2}}\int_{\mathbb{R}^{2}}\log \left(
|x-y|\right) |(V_{n}(x))_{t_{0}}|^{2}|(V_{n}(y))_{t_{0}}|^{2}dxdy \\
&&+\frac{\alpha \rho ^{2}}{4(p-2)} \\
&=&\frac{(p-4)t_{0}^{2}}{2(p-2)}\int_{\mathbb{R}^{2}}|\nabla V_{n}|^{2}dx+%
\frac{\alpha }{4}\int_{\mathbb{R}^{2}}\int_{\mathbb{R}^{2}}\log \left(
|x-y|\right) |V_{n}(x)|^{2}|V_{n}(y)|^{2}dxdy \\
&&-\frac{\alpha \rho ^{2}}{4}\log t_{0}+\frac{\alpha \rho ^{2}}{4(p-2)} \\
&\leq &\frac{2^{\frac{2}{p-4}}(p-4)t_{0}^{2}}{p-2}\int_{\mathbb{R}%
^{2}}|\nabla \psi |^{2}dx+\frac{\alpha }{8}\int_{\mathbb{R}^{2}}\int_{%
\mathbb{R}^{2}}\log \left( |x-y|\right) |\psi (x)|^{2}|\psi (y)|^{2}dxdy \\
&&-\frac{\alpha }{4}\rho ^{2}\log t_{0}+\frac{\alpha \rho ^{2}}{4(p-2)}+%
\frac{\alpha \log \left( n^{2}-1\right) \rho ^{2}}{8}-\frac{\alpha (p-2)}{%
8(p-4)}\rho ^{2}\log 2 \\
&\leq &\frac{2^{\frac{2}{p-4}}(p-4)}{p-2}t_{0}^{2}\lambda _{1}(\Omega )\rho -%
\frac{\alpha }{8}\mathcal{C}\mathcal{C}_{8/3}^{3/2}(\lambda _{1}(\Omega
))^{1/2}\rho ^{2} \\
&&-\frac{\alpha }{4}\rho ^{2}\log t_{0}+\frac{\alpha \rho ^{2}}{4(p-2)}+%
\frac{\alpha \log \left( n^{2}-1\right) \rho ^{2}}{8}-\frac{\alpha (p-2)}{%
8(p-4)}\rho ^{2}\log 2 \\
&\rightarrow &-\infty \ \text{as}\ n\rightarrow \infty ,
\end{eqnarray*}%
leading to $M_{\rho }=-\infty $.

Next we claim that $\overline{M}_{\rho }=+\infty $. Set $W_{n}:=%
\sum_{i=1}^{n}n^{\frac{1}{p-4}}\psi \left( n^{\frac{p-2}{2(p-4)}}\left(
x+in^{2}e\right) \right) $ for $n\in
\mathbb{N}
$. Then for $t>0,$ by (\ref{E9}) one has
\begin{eqnarray}
&&\int_{\mathbb{R}^{2}}|\left( W_{n}\right) _{t}|^{2}dx=\int_{\mathbb{R}%
^{2}}|\psi |^{2}dx=\rho ,  \notag \\
&&\int_{\mathbb{R}^{2}}|\nabla \left( W_{n}\right) _{t}|^{2}dx=n^{\frac{p-2}{%
p-4}}t^{2}\int_{\mathbb{R}^{2}}|\nabla \psi |^{2}dx=n^{\frac{p-2}{p-4}%
}t^{2}\lambda _{1}(\Omega )\rho ,  \notag \\
&&\int_{\mathbb{R}^{2}}|\left( W_{n}\right) _{t}|^{p}dx=n^{\frac{p-2}{p-4}%
}t^{p-2}\int_{\mathbb{R}^{2}}|\psi |^{p}dx\leq n^{\frac{p-2}{p-4}}t^{2}%
\mathcal{C}_{p}\left[ \lambda _{1}\left( \Omega \right) \rho \right]
^{(p-2)/2}\rho  \label{E12}
\end{eqnarray}%
and
\begin{eqnarray}
&&\int_{\mathbb{R}^{2}}\int_{\mathbb{R}^{2}}\log (|x-y|)|\left(
W_{n}(x)\right) _{t}|^{2}|\left( W_{n}(y)\right) _{t}|^{2}dxdy  \notag \\
&=&\int_{\mathbb{R}^{2}}\int_{\mathbb{R}^{2}}\log
(|x-y|)|W_{n}(x)|^{2}|W_{n}(y)|^{2}dxdy-\left( \int_{\mathbb{R}%
^{2}}|W_{n}|^{2}dx\right) ^{2}\log t  \notag \\
&=&n^{-2}\int_{\mathbb{R}^{2}}\int_{\mathbb{R}^{2}}\log
(|x-y|)|\sum_{i=1}^{n}\psi \left( x+in^{2}e\right) |^{2}|\sum_{j=1}^{n}\psi
\left( y+jn^{2}e\right) |^{2}dxdy  \notag \\
&&-\frac{p-2}{2(p-4)}\rho ^{2}\log n-\rho ^{2}\log t  \notag \\
&=&n^{-1}\int_{\mathbb{R}^{2}}\int_{\mathbb{R}^{2}}\log (|x-y|)|\psi
(x)|^{2}|\psi (y)|^{2}dxdy  \notag \\
&&+n^{-2}\sum_{i\neq j}^{n}\int_{\mathbb{R}^{2}}\int_{\mathbb{R}^{2}}\log
(|x-y+(j-i)n^{2}e|)|\psi (x)|^{2}|\psi (y)|^{2}dxdy  \notag \\
&&-\frac{p-2}{2(p-4)}\rho ^{2}\log n-\rho ^{2}\log t  \notag \\
&\leq &n^{-1}\int_{\mathbb{R}^{2}}\int_{\mathbb{R}^{2}}\log (|x-y|)|\psi
(x)|^{2}|\psi (y)|^{2}dxdy  \notag \\
&&+\rho ^{2}\log (n^{3}+1)\frac{n^{2}-n}{n^{2}}-\frac{p-2}{2(p-4)}\rho
^{2}\log n-\rho ^{2}\log t  \notag \\
&\leq &n^{-1}\rho ^{2}\log 3+\rho ^{2}\log (n^{3}+1)-\frac{p-2}{2(p-4)}\rho
^{2}\log n-\rho ^{2}\log t.  \label{E13}
\end{eqnarray}%
It follows from (\ref{E12}) that
\begin{eqnarray*}
P\left( \left( W_{n}\right) _{n^{-1/(p-4)}}\right) &=&n\int_{\mathbb{R}%
^{2}}|\nabla \psi |^{2}dx-\frac{p-2}{p}\int_{\mathbb{R}^{2}}|\psi |^{p}dx+%
\frac{\alpha \rho ^{2}}{4} \\
&\geq &n\lambda _{1}(\Omega )\rho -\frac{p-2}{p}\mathcal{C}_{p}\left[
\lambda _{1}\left( \Omega \right) \rho \right] ^{(p-2)/2}\rho +\frac{\alpha
\rho ^{2}}{4} \\
&\geq &0,
\end{eqnarray*}%
for $n>0$ large enough. Thus, combining with Lemma \ref{L3.1}, there exists $%
t_{1,n}\geq n^{-1/(p-4)}$ such that $P\left( \left( W_{n}\right)
_{t_{1,n}}\right) =0$. Using this, together with (\ref{E12})--(\ref{E13}),
yields
\begin{eqnarray*}
&&J\left( \left( W_{n}\right) _{t_{1,n}}\right) \\
&=&\frac{p-4}{2(p-2)}\int_{\mathbb{R}^{2}}|\nabla (W_{n})_{t_{1,n}}|^{2}dx+%
\frac{\alpha }{4}\int_{\mathbb{R}^{2}}\int_{\mathbb{R}^{2}}\log \left(
|x-y|\right) |(W_{n}(x))_{t_{1,n}}|^{2}|(W_{n}(y))_{t_{1,n}}|^{2}dxdy \\
&&+\frac{\alpha \rho ^{2}}{4(p-2)} \\
&\geq &\frac{p-4}{2(p-2)}n^{\frac{p-2}{p-4}}t_{1,n}^{2}\int_{\mathbb{R}%
^{2}}|\nabla \psi |^{2}dx+\frac{\alpha \rho ^{2}}{4(p-2)} \\
&&+\frac{\alpha }{4}n^{-1}\int_{\mathbb{R}^{2}}\int_{\mathbb{R}^{2}}\log
(|x-y|)|\psi (x)|^{2}|\psi (y)|^{2}dxdy \\
&&+\frac{\alpha }{4}\rho ^{2}\log (n^{3}+1)\frac{(n^{2}-n)}{n^{2}}-\frac{%
\alpha }{4}\left( \frac{p-2}{2(p-4)}\rho ^{2}\log n+\rho ^{2}\log
t_{1,n}\right) \\
&\geq &\frac{p-4}{2(p-2)}\lambda _{1}(\Omega )\rho n+\frac{\alpha \rho ^{2}}{%
4(p-2)}+\frac{\alpha }{4}\rho ^{2}\log (n^{3}+1)+\frac{\alpha }{4}n^{-1}\rho
^{2}\log 3 \\
&\rightarrow &+\infty \ \text{as}\ n\rightarrow \infty ,
\end{eqnarray*}%
which shows that $\overline{M}_{\rho }=+\infty $. The proof is complete.
\end{proof}

\textbf{Proof of Theorem \ref{T1.1}:} It is a direct
conclusion of Theorem \ref{L3.2}.

\section{The existence and asymptotic behavior of solutions}

\subsection{The local minimizer}

By the Gagliardo-Nirenberg and the H\"{o}lder equalities, we have
\begin{equation*}
\widetilde{J}_{R}(u)\geq \frac{1}{2}\Vert \nabla u\Vert _{2}^{2}-\frac{1}{p}%
\mathcal{C}_{p}\Vert \nabla u\Vert _{2}^{p-2}\rho +\frac{\alpha \rho ^{2}}{4}%
\log (1+R).
\end{equation*}%
Define a function $f:\mathbb{R}^{+}\rightarrow \mathbb{R}$ given by%
\begin{equation}
f(x)=\frac{1}{2}x^{2}-\frac{1}{p}\mathcal{C}_{p}x^{p-2}\rho +\frac{\alpha
\rho ^{2}}{4}\log (1+R)\text{ for }x>0.  \label{E14}
\end{equation}%
Set
\begin{equation*}
R_{0}:=(\lambda _{1}(\Omega )\rho )^{1/2}\left( \frac{(p-2)\rho \mathcal{C}%
_{p}}{p}\right) ^{\frac{1}{p-4}}
\end{equation*}%
and
\begin{equation*}
\alpha _{R,\rho }^{0}:=\frac{2R^{-2}\lambda _{1}(\Omega )-\frac{4}{p}%
\mathcal{C}_{p}\left( R^{-2}\lambda _{1}(\Omega )\rho \right) ^{(p-2)/2}}{%
\rho \log (1+R)}.
\end{equation*}%
Then we have the following lemma.

\begin{lemma}
\label{L4.1}Let $\rho >0$, $R>R_{0}$ and $0<|\alpha |<\alpha _{R,\rho }^{0}$%
. Then the function $f$ has a global strict maximum point $x_{\ast
}>R^{-1}\left( \lambda _{1}(\Omega )\rho \right) ^{1/2}$ satisfying $f\left(
x_{\ast }\right) >f\left( R^{-1}\left( \lambda _{1}(\Omega )\rho \right)
^{1/2}\right) >0$. Moreover, $f$ is increasing on $0<x<x_{\ast }$ and is
decreasing on $x>x_{\ast }$.
\end{lemma}

\begin{proof}
By calculating, we have
\begin{equation*}
f^{\prime }(x)=x-\frac{p-2}{p}\mathcal{C}_{p}x^{p-3}\rho ,
\end{equation*}%
which implies that $x_{\ast }:=\left[ \frac{p}{\rho \mathcal{C}_{p}(p-2)}%
\right] ^{\frac{1}{p-4}}$ is the global strict maximum point of $f$, i.e.,
\begin{equation}
f(x_{\ast })=\max_{x>0}f(x)=\frac{p-4}{2(p-2)}\left( \frac{p}{(p-2)\rho
\mathcal{C}_{p}}\right) ^{\frac{2}{p-4}}+\frac{\alpha }{4}\rho ^{2}\log
(1+R).  \label{E15}
\end{equation}%
Moreover, $f$ is increasing on $0<x<x_{\ast }$ and is decreasing on $%
x>x_{\ast }$. Since $R>R_{0}$, we have $x_{\ast }>R^{-1}\left( \lambda
_{1}(\Omega )\rho \right) ^{1/2}$. Moreover, for $R>R_{0}$ and $0<|\alpha
|<\alpha _{R,\rho }^{0}$, we obtain
\begin{equation*}
f\left( R^{-1}\left( \lambda _{1}(\Omega )\rho \right) ^{1/2}\right) =\frac{1%
}{2}R^{-2}\lambda _{1}(\Omega )\rho -\frac{1}{p}\mathcal{C}_{p}\left(
R^{-2}\lambda _{1}(\Omega )\rho \right) ^{(p-2)/2}\rho +\frac{\alpha }{4}%
\rho ^{2}\log (1+R)>0.
\end{equation*}%
Thus, $f\left( x_{\ast }\right) >f\left( R^{-1}\left( \lambda _{1}(\Omega
)\rho \right) ^{1/2}\right) >0$. The proof is complete.
\end{proof}

Define
\begin{equation*}
Q_{\rho ,R}=\left\{ u\in \mathcal{S}_{\Omega _{R},\alpha }:\Vert \nabla
u\Vert _{2,\Omega _{R}}\leq x_{\ast }\right\} .
\end{equation*}

\begin{proposition}
\label{L4.2} For any $\rho >0$, there exist $R_{\rho }\geq R_{0}$ and $%
\alpha _{R,\rho }^{\ast }\leq \alpha _{R,\rho }^{0}$ such that for any $%
R>R_{\rho }\ $and $0<|\alpha |<\alpha _{R,\rho }^{\ast }$, there hold $%
Q_{\rho ,R}\neq \emptyset $ and
\begin{equation*}
C_{R,\rho }:=\inf_{u\in Q_{\rho ,R}}\widetilde{J}_{R}(u)>0
\end{equation*}%
is achieved at $u_{R}^{0}\in Q_{\rho ,R}\setminus \partial Q_{\rho ,R}$.
Moreover, $u_{R}^{0}$ is a solution of problem (\ref{E5}) for some Lagrange
multiplier $\lambda _{R,\rho }\in \mathbb{R}$.
\end{proposition}

\begin{proof}
Let $\psi \in \mathcal{S}_{\Omega ,\rho }$ be the positive normalized
eigenfunction $-\Delta $ with Dirichlet boundary condition in $\Omega $
associated to $\lambda _{1}(\Omega )$. It follows from the Poincar\'{e}
inequality that
\begin{equation}
\int_{\Omega _{R}}|\nabla u|^{2}dx\geq R^{-2}\lambda _{1}(\Omega )\rho \
\text{for any}\ u\in \mathcal{S}_{\Omega _{R},\alpha }.  \label{E16}
\end{equation}%
By virtue of (\ref{E16}) and the fact that $x_{\ast }$ is independent of $R$%
, we obtain that $Q_{\rho ,R}\neq \emptyset $ if and only if $R\geq R_{0}$.

Next, we claim that $C_{R,\rho }>0$ is achieved. For any $u\in Q_{\rho ,R},$
it follows from (\ref{E16}) that
\begin{equation*}
x_{\ast }^{2}>\int_{\Omega _{R}}|\nabla u|^{2}dx\geq R^{-2}\lambda
_{1}(\Omega )\rho .
\end{equation*}%
So by (\ref{E14}) and Lemma \ref{L4.1}, we get $C_{R,\rho }>0$. By the
Ekeland variational principle, there exists a sequence $\{u_{n}\}\subset
Q_{\rho ,R}$ such that
\begin{equation*}
\widetilde{J}_{R}(u_{n})\rightarrow C_{R,\rho }\ \text{and}\ \widetilde{J}%
^{\prime }(u_{n})|_{T_{u_{n}}Q_{\rho ,R}}\rightarrow 0\ \text{as}\
n\rightarrow \infty .
\end{equation*}%
Since $\Vert \nabla u_{n}\Vert _{2,\Omega _{R}}<x_{\ast }$, there exists $%
u_{R}^{0}\in H_{0}^{1}\left( \Omega _{R}\right) $ such that up to a
subsequence,
\begin{equation*}
u_{n}\rightharpoonup u_{R}^{0}\ \text{in}\ H_{0}^{1}\left( \Omega
_{R}\right) \ \text{and }u_{n}\rightarrow u_{R}^{0}\ \text{in}\ L^{q}\left(
\Omega _{R}\right) \ \text{for any}\ q\geq 1.
\end{equation*}%
Moreover, using the Fatou's lemma, one has $\Vert \nabla u_{R}^{0}\Vert
_{2,\Omega _{R}}\leq \liminf_{n\rightarrow \infty }\Vert \nabla u_{n}\Vert
_{2,\Omega _{R}}\leq x_{\ast }$, which implies that $u_{R}^{0}\in Q_{\rho
,R} $ and so $\widetilde{J}_{R}\left( u_{R}^{0}\right) \geq C_{R,\rho }$.
Denote $w_{n}=u_{n}-u_{R}^{0}$, then we calculate that $\Vert w_{n}\Vert
_{p,\Omega _{R}}^{p}=o_{n}(1)$ and
\begin{eqnarray*}
&&\left\vert \int_{\Omega _{R}}\int_{\Omega _{R}}\log \left( |x-y|\right)
|w_{n}(x)|^{2}|w_{n}(y)|^{2}dxdy\right\vert \\
&\leq &\log \left( 1+R\right) \int_{\Omega _{R}}\int_{\Omega
_{R}}|w_{n}(x)|^{2}|w_{n}(y)|^{2}dxdy+\mathcal{C}\left( \int_{\Omega
_{R}}|w_{n}|^{8/3}dx\right) ^{3/2} \\
&=&o_{n}(1).
\end{eqnarray*}%
This shows that
\begin{equation*}
\widetilde{J}_{R}(u_{n})=\widetilde{J}_{R}\left( u_{R}^{0}\right) +\frac{1}{2%
}\Vert \nabla w_{n}\Vert _{2,\Omega _{R}}^{2}+o_{n}(1)\geq C_{R,\rho
}+o_{n}(1),
\end{equation*}%
which implies that $u_{n}\rightarrow u_{R}^{0}$ in $H_{0}^{1}\left( \Omega
_{R}\right) $.

Finally, we claim that $u_{R}^{0}\in Q_{\rho ,R}\setminus \partial Q_{\rho
,R}$. For each $x\in \Omega $, set $\psi _{R}(x):=R^{-1}\psi \left(
R^{-1}x\right) $. Obviously, $\psi _{R}\in Q_{\rho ,R}$. If
\begin{equation}
\widetilde{J}_{R}(u)|_{\partial Q_{\rho ,R}}>\widetilde{J}_{R}\left( \psi
_{R}\right) ,  \label{E17}
\end{equation}%
then%
\begin{equation*}
C_{R,\rho }<\inf_{u\in \partial Q_{\rho ,R}}\widetilde{J}_{R}(u)
\end{equation*}%
and $u_{R}^{0}\in Q_{\rho ,R}\setminus \partial Q_{\rho ,R}$. So we only
need to prove that (\ref{E17}) holds. Note that
\begin{equation*}
\int_{\Omega }|\nabla \psi |^{2}dx=\lambda _{1}(\Omega )\rho \ \text{and}\
\int_{\Omega }|\psi |^{p}dx\geq \rho ^{\frac{p}{2}}|\Omega |^{-\frac{p-2}{2}%
}.
\end{equation*}%
Then
\begin{equation}
\int_{\Omega _{R}}|\nabla \psi _{R}|^{2}dx=R^{-2}\int_{\Omega }|\nabla \psi
|^{2}dx=R^{-2}\lambda _{1}(\Omega )\rho  \label{E18}
\end{equation}%
and
\begin{equation}
\int_{\Omega _{R}}|\psi _{R}|^{p}dx=R^{-(p-2)}\int_{\Omega }|\psi
|^{p}dx\geq R^{-(p-2)}\rho ^{p/2}|\Omega |^{-(p-2)/2}.  \label{E19}
\end{equation}%
By (\ref{E7}) and (\ref{E18}), we have
\begin{eqnarray}
&&\int_{\Omega _{R}}\int_{\Omega _{R}}\log (|x-y|)|\psi _{R}(x)|^{2}|\psi
_{R}(y)|^{2}dxdy  \notag \\
&=&\int_{\Omega _{R}}\int_{\Omega _{R}}\log (1+|x-y|)|\psi _{R}(x)|^{2}|\psi
_{R}(y)|^{2}dxdy  \notag \\
&&-\int_{\Omega _{R}}\int_{\Omega _{R}}\log \left( 1+\frac{1}{|x-y|}\right)
|\psi _{R}(x)|^{2}|\psi _{R}(y)|^{2}dxdy  \notag \\
&\geq &-\mathcal{C}\left( \int_{\Omega _{R}}|\psi _{R}|^{8/3}dx\right) ^{3/2}
\notag \\
&\geq &-\mathcal{CC}_{8/3}^{3/2}\rho ^{3/2}\left( \int_{\Omega _{R}}|\nabla
\psi _{R}|^{2}dx\right) ^{1/2}  \notag \\
&=&-\mathcal{CC}_{8/3}^{3/2}R^{-1}\lambda _{1}(\Omega )^{1/2}\rho ^{2}.
\label{E20}
\end{eqnarray}%
It follows from (\ref{E18})--(\ref{E20}) that
\begin{eqnarray}
C_{R,\rho } &\leq &\widetilde{J}_{R}\left( \psi _{R}\right)  \notag \\
&=&\frac{1}{2}\int_{\Omega _{R}}|\nabla \psi _{R}|^{2}dx+\frac{\alpha }{4}%
\int_{\Omega _{R}}\int_{\Omega _{R}}\log (|x-y|)|\psi _{R}(x)|^{2}|\psi
_{R}(y)|^{2}dxdy  \notag \\
&&-\frac{1}{p}\int_{\Omega }|\psi _{R}|^{p}dx  \notag \\
&\leq &\frac{1}{2}R^{-2}\lambda _{1}(\Omega )\rho -\frac{1}{p}R^{-(p-2)}\rho
^{p/2}|\Omega |^{-(p-2)/2}-\frac{\alpha }{4}\mathcal{CC}_{8/3}^{3/2}R^{-1}%
\lambda _{1}(\Omega )^{1/2}\rho ^{2}.  \label{E21}
\end{eqnarray}%
By (\ref{E14})--(\ref{E15}), one has
\begin{equation}
\widetilde{J}_{R}(u)|_{\partial Q_{\rho ,R}}\geq f(x_{\ast })=\frac{p-4}{%
2(p-2)}\left( \frac{p}{(p-2)\rho \mathcal{C}_{p}}\right) ^{\frac{2}{p-4}}+%
\frac{\alpha }{4}\rho ^{2}\log (1+R)>0.  \label{E22}
\end{equation}%
Combining with (\ref{E21}) and (\ref{E22}), we conclude that if the
following inequality
\begin{eqnarray}
&&\frac{p-4}{2(p-2)}\left( \frac{p}{(p-2)\rho \mathcal{C}_{p}}\right) ^{%
\frac{2}{p-4}}+\frac{\alpha }{4}\rho ^{2}\log (1+R)  \notag \\
&>&\frac{1}{2}R^{-2}\lambda _{1}(\Omega )\rho -\frac{1}{p}R^{-(p-2)}\rho
^{p/2}|\Omega |^{-(p-2)/2}-\frac{\alpha }{4}\mathcal{CC}_{8/3}^{3/2}R^{-1}%
\lambda _{1}(\Omega )^{1/2}\rho ^{2}  \label{E23}
\end{eqnarray}%
holds, then (\ref{E17}) is true. In fact, to prove (\ref{E23}), we only need%
\begin{equation*}
0<|\alpha |<\frac{\frac{2(p-4)}{p-2}\left( \frac{p}{(p-2)\rho \mathcal{C}_{p}%
}\right) ^{\frac{2}{p-4}}-2R^{-2}\lambda _{1}(\Omega )\rho +\frac{4}{p}%
R^{-(p-2)}\rho ^{p/2}|\Omega |^{-(p-2)/2}}{\rho ^{2}\log (1+R)+\mathcal{CC}%
_{8/3}^{3/2}R^{-1}\lambda _{1}(\Omega )^{1/2}\rho ^{2}}=:\alpha _{R,\rho
}^{1}
\end{equation*}%
and%
\begin{equation*}
\frac{p-4}{2(p-2)}\left( \frac{p}{(p-2)\rho \mathcal{C}_{p}}\right) ^{\frac{2%
}{p-4}}-\frac{1}{2}R^{-2}\lambda _{1}(\Omega )\rho +\frac{1}{p}%
R^{-(p-2)}\rho ^{p/2}|\Omega |^{-(p-2)/2}>0
\end{equation*}%
by controlling $R>0$ large enough. Therefore, there exists $R_{\rho }\geq
R_{0}$ independent on $\alpha $ such that for any $R>R_{\rho }$ and $%
0<|\alpha |<\alpha _{R,\rho }^{\ast }:=\min \{\alpha _{R,\rho }^{0},\alpha
_{R,\rho }^{1}\}$, we arrive at the conclusion. The proof is complete.
\end{proof}

Similar to the proof of \cite[Lemma 2.7]{CJ2019}, any weak solution $u\in
H_{0}^{1}\left( \Omega _{R}\right) $ of Problem (\ref{E5}) satisfies the
Pohozaev identity as follows:
\begin{equation}
P_{\Omega _{R}}(u):=\Vert \nabla u\Vert _{2,\Omega _{R}}^{2}-\frac{p-2}{p}%
\Vert u\Vert _{p,\Omega _{R}}^{p}-\frac{\alpha \rho ^{2}}{4}-\frac{1}{2}%
\int_{\partial \Omega _{R}}|\nabla u|^{2}x\cdot nd\sigma =0.  \label{E24}
\end{equation}%
If $\Omega $ is a star-shaped bounded domain, then
\begin{equation}
\Vert \nabla u\Vert _{2,\Omega _{R}}^{2}-\frac{p-2}{p}\Vert u\Vert
_{p,\Omega _{R}}^{p}-\frac{\alpha \rho ^{2}}{4}=\frac{1}{2}\int_{\partial
\Omega _{R}}|\nabla u|^{2}x\cdot nd\sigma \geq 0,  \label{E25}
\end{equation}%
and%
\begin{equation*}
\alpha _{R,\rho }^{\ast }\log (1+R)\leq \alpha _{R,\rho }^{0}\log
(1+R)=2R^{-2}\lambda _{1}(\Omega )\rho ^{-1}-\frac{4}{p}\mathcal{C}_{p}\rho
\left( R^{-2}\lambda _{1}(\Omega )\rho \right) ^{(p-2)/2}\rho
^{-1}\rightarrow 0^{+}\text{ as }R\rightarrow \infty ,
\end{equation*}%
since $\alpha _{R,\rho }^{0}\rightarrow 0$ as $R\rightarrow \infty .$ This
shows that%
\begin{equation}
\lim_{R\rightarrow \infty }\alpha \log (1+R)=0\text{ for }0<|\alpha |<\alpha
_{R,\rho }^{\ast }.  \label{E26}
\end{equation}%
It follows from (\ref{E21}) and (\ref{E26}) that
\begin{eqnarray}
0<C_{R,\rho }\leq \widetilde{J}_{R}(\psi _{R}) &\leq &\frac{1}{2}%
R^{-2}\lambda _{1}(\Omega )\rho -\frac{1}{p}R^{-(p-2)}\rho ^{p/2}|\Omega
|^{-(p-2)/2}-\frac{\alpha }{4}\mathcal{CC}_{8/3}^{3/2}R^{-1}\lambda
_{1}(\Omega )^{1/2}\rho ^{2}  \notag \\
&\rightarrow &0^{+}\ \text{as}\ R\rightarrow \infty .  \label{E27}
\end{eqnarray}

Define
\begin{equation*}
m_{\rho ,\Omega _{R}}=\inf_{u\in \mathcal{P}_{\Omega _{R}}}\widetilde{J}%
_{R}(u).
\end{equation*}%
where%
\begin{equation*}
\mathcal{P}_{\Omega _{R}}:=\left\{ u\in \mathcal{S}_{\rho ,\Omega
_{R}}:P_{\Omega _{R}}(u)=0\right\} .
\end{equation*}%
Obviously, $m_{\rho ,\Omega _{R}}\leq C_{R,\rho }$.

\begin{lemma}
\label{L4.3} Let $\rho >0$ be fixed. There exists $R_{\ast }\geq R_{\rho }$
such that for any $R\geq R_{\ast }$ and $0<|\alpha |<\alpha _{R,\rho }^{\ast
}$, it holds that $m_{\rho ,\Omega _{R}}=C_{R,\rho }$.
\end{lemma}

\begin{proof}
Let $u$ be the solution of Problem (\ref{E5}) such that $\widetilde{J}%
_{R}(u)=m_{\rho ,\Omega _{R}}$. Using (\ref{E25})--(\ref{E27}), we have
\begin{eqnarray*}
\int_{\Omega _{R}}|\nabla u|^{2}dx &\leq &\frac{2(p-2)}{p-4}\left(
\widetilde{J}_{R}(u)-\frac{\alpha \rho ^{2}}{4(p-2)}-\frac{\alpha }{4}\log
(1+R)\rho ^{2}\right) \\
&\leq &\frac{2(p-2)}{p-4}\left( C_{R,\rho }-\frac{\alpha \rho ^{2}}{4(p-2)}-%
\frac{\alpha }{4}\log (1+R)\rho ^{2}\right) \\
&\rightarrow &0\ \text{as}\ R\rightarrow \infty ,
\end{eqnarray*}%
which implies that there exists $R_{\ast }\geq R_{\rho }$ such that $%
\int_{\Omega _{R}}|\nabla u|^{2}dx<x_{\ast }^{2}$ for any $R>R_{\ast }$. The
proof is complete.
\end{proof}

\textbf{Proof of Theorem \ref{T1.3} $(i)-(ii)$:}
$(i)$ By Proposition \ref{L4.2}, we get that Theorem \ref{T1.3} $(i)$ holds.\newline
$(ii)$ By Lemma \ref{L4.3}, we arrive the conclusion of Theorem \ref{T1.3} $(ii)$.

\subsection{The mountain-pass solution}

To obtain a mountain-pass solution, we introduce a modified problem related
to Problem (\ref{E5}):
\begin{equation}
\left\{
\begin{array}{ll}
-\Delta u+\lambda u+\alpha \left( \log |\cdot |\ast |u|^{2}\right)
u=s|u|^{p-2}u, & \ \text{in}\ \Omega _{R} \\
\int_{\Omega _{R}}|u|^{2}dx=\rho >0, &
\end{array}%
\right.  \label{E28}
\end{equation}%
where $s\in \lbrack 1/2,1]$. Solutions of Problem (\ref{E28}) correspond to
critical points of the energy function $\widetilde{J}_{R,s}:\mathcal{S}%
_{\Omega _{R},\rho }\rightarrow \mathbb{R}$ defined by
\begin{equation*}
\widetilde{J}_{R,s}(u)=\frac{1}{2}\int_{\Omega _{R}}|\nabla u|^{2}dx+\frac{%
\alpha }{4}\int_{\Omega _{R}}\int_{\Omega _{R}}\log \left( |x-y|\right)
|u(x)|^{2}|u(y)|^{2}dxdy-\frac{s}{p}\int_{\Omega _{R}}|u|^{p}dx
\end{equation*}%
on the constraint $\mathcal{S}_{\Omega _{R},\rho }$. Similar to the argument
of Theorem \ref{T1.3} $(i),$ for any $\rho >0$, $R>R_{\rho }$ and $0<|\alpha
|<\alpha _{R,\rho }^{\ast }$, Problem (\ref{E28}) has a local minimizer $%
v_{R,s}\in Q_{\rho ,R}$ satisfying%
\begin{equation*}
\widetilde{J}_{R,s}(v_{R,s})=\inf_{u\in Q_{\rho ,R}}\widetilde{J}%
_{R,s}(u)=:C_{R,s,\rho }.
\end{equation*}%
Next, we consider the geometry structure of the energy functional $%
\widetilde{J}_{R,s}$.

\begin{lemma}
\label{L4.4} For any $R>R_{\rho }$ and $0<|\alpha |<\alpha _{R,\rho }^{\ast
} $, it holds that
\begin{equation*}
0<f\left( x_{\ast }\right) \leq \overline{C}_{R,s,\rho
}:=\inf\limits_{\gamma \in \Gamma _{R,\rho }}\sup\limits_{t\in (0,\infty )}%
\widetilde{J}_{R,s}\left( \gamma (t)\right) \leq \max_{t>1}\widetilde{J}%
_{R,s}\left( (v_{R,s})_{t}\right) ,
\end{equation*}%
where%
\begin{eqnarray*}
\Gamma _{R,\rho } &:=&\left\{ \gamma \in C\left( [0,\infty
),H_{0}^{1}(\Omega _{R})\right) :\gamma (0)=v_{R,s},\exists t_{\gamma }>0%
\text{ }s.t.\ \Vert \nabla \gamma (t)\Vert _{2}>x_{\ast}\right. \\
&&\left. \text{and}\ \widetilde{J}_{R,s}\left( \gamma (t)\right) <0\ \text{%
for any}\ t>t_{\gamma }\right\} .
\end{eqnarray*}
\end{lemma}

\begin{proof}
Denote $(v_{R,s})_{t}=tv_{R,s}(tx)$ for any $t\geq 1$. Then $%
(v_{R,s})_{t}\in \mathcal{S}_{\Omega _{R},\rho }$. A direct calculation
shows that
\begin{eqnarray*}
\widetilde{J}_{R,s}\left( (v_{R,s})_{t}\right) &=&\frac{t^{2}}{2}%
\int_{\Omega _{R}}|\nabla v_{R,s}|^{2}dx+\frac{\alpha }{4}\int_{\Omega
_{R}}\int_{\Omega _{R}}|v_{R,s}(x)|^{2}|u_{R}(y)|^{2}\log \left(
|x-y|\right) dxdy \\
&&-\frac{\alpha \rho ^{2}}{4}\log t-\frac{t^{p-2}s}{p}\int_{\Omega
_{R}}|v_{R,s}|^{p}dx \\
&\rightarrow &-\infty \ \text{as}\ t\rightarrow \infty ,
\end{eqnarray*}%
which indicates that there exists $t_{1}>1$ such that $\widetilde{J}%
_{R}\left( (v_{R,s})_{t_{1}}\right) =\max_{t\geq 1}\widetilde{J}_{R}\left(
(v_{R,s})_{t}\right) \geq \widetilde{J}_{R}\left( v_{R,s}\right) .$
Moreover, there exists $t_{\ast }>t_{1}$ such that $\widetilde{J}_{R}\left(
(v_{R,s})_{t+1}\right) <0$ and $\Vert \nabla (v_{R,s})_{t+1}\Vert
_{2}>x_{\ast }$ for any $t>t_{\ast }$. Set $\gamma (t):=(v_{R,s})_{t+1}$ for
$t\geq 0$. Obviously, $\gamma (t)\in \Gamma _{R,\rho }$. Note that for any $%
\gamma \in \Gamma _{R,\rho }$ and $t>t_{\ast }$, there holds $\Vert \nabla
\gamma (0)\Vert _{2}<x_{\ast }<\Vert \nabla \gamma (t)\Vert _{2}$, which
shows that
\begin{equation*}
\max_{t>1}\widetilde{J}_{R,s}\left( (v_{R,s})_{t}\right) \geq
\inf\limits_{\gamma \in \Gamma _{R,\rho }}\sup_{t\in \lbrack 0,\infty )}%
\widetilde{J}_{R,s}\left( \gamma (t)\right) \geq \max_{x>0}f(x)=f\left(
x_{\ast }\right) >0.
\end{equation*}%
The proof is complete.
\end{proof}

\begin{theorem}
\label{L4.5} (\cite[Theorem 1.5]{BCJS}) Let $(E,\langle \cdot ,\cdot \rangle
)$ and $(H,\langle \cdot ,\cdot \rangle )$ be two infinite-dimensional
Hilbert spaces and assume that there are continuous injections
\begin{equation*}
E\hookrightarrow H\hookrightarrow E^{\prime }.
\end{equation*}%
Let
\begin{equation*}
\Vert u\Vert _{2}^{2}=\langle u,u\rangle ,\ |u|^{2}=(u,u)\ \text{for}\ u\in E
\end{equation*}%
and
\begin{equation*}
S_{\mu }=\{u\in E:|u|^{2}=\mu \},\ T_{\mu }S_{\mu }=\{v\in E:(u,v)=0\}\
\text{for}\ \mu \in (0,+\infty ).
\end{equation*}%
Let $I\subset (0,+\infty )$ be an interval and consider a family of $C^{2}$
functional $\Phi _{a}:E\rightarrow \mathbb{R}$ of the form
\begin{equation*}
\Phi _{a}=A(u)-aB(u)\ \text{for}\ \rho \in I,
\end{equation*}%
with $B(u)\geq 0$ for every $u\in E$, and
\begin{equation*}
A(u)\rightarrow \infty \ \text{or}\ B(u)\rightarrow +\infty \ \text{as}\
u\in E\ \text{and}\ \Vert u\Vert \rightarrow +\infty .
\end{equation*}%
Moreover, $\Phi _{a}^{\prime }$ and $\Phi _{a}^{\prime \prime }$ are $\tau $%
-H\"{o}lder continuous with $\tau \in (0,1]$, on bounded sets in the
following sense: for every $R>0$ there exists $M=M(R)>0$ such that
\begin{equation*}
\Vert \Phi _{a}^{\prime }(u)-\Phi _{a}^{\prime }(v)\Vert \leq M\Vert
u-v\Vert ^{\tau }\ \text{and}\ \Vert \Phi _{a}^{\prime \prime }(u)-\Phi
_{a}^{\prime \prime }(v)\Vert \leq M\Vert u-v\Vert ^{\tau }
\end{equation*}%
for every $u,v\in B(0,R)$. Finally, suppose that there exist $w_{1},w_{2}\in
S_{\mu }$ independent of $\rho $ such that
\begin{equation*}
c_{a}:=\inf\limits_{\gamma \in \Gamma }\max\limits_{t\in \lbrack 0,1]}\Phi
_{a}(\gamma (t))>\max \{\Phi _{a}(w_{1}),\Phi _{a}(w_{2})\}\ \text{for all}\
a\in I,
\end{equation*}%
where
\begin{equation*}
\Gamma :=\{\gamma \in C([0,1],S_{\mu }):\gamma (0)=w_{1},\ \gamma
(1)=w_{2}\}.
\end{equation*}%
Then for almost every $\rho \in I,$ there exists a sequence $%
\{u_{n}\}\subset S_{\mu }$ such that\newline
$(i)$ $\Phi _{a}(u_{n})\rightarrow c_{a};$ $(ii)$ $\Phi _{a}^{\prime
}|_{S_{\mu }}(u_{n})\rightarrow 0;$ $(iii)$ $\{u_{n}\}$ is bounded in $E.$
\end{theorem}

\begin{proposition}
\label{L4.6} For any $R>R_{\rho }$ and $0<|\alpha |<\alpha _{R,\rho }^{\ast
} $, Problem (\ref{E28}) has a mountain-pass type solution for almost every $%
s\in \lbrack 1/2,1]$.
\end{proposition}

\begin{proof}
According to Lemma \ref{L4.4} and Theorem \ref{L4.5}, we conclude that for
almost every $s\in \lbrack 1/2,1]$, there exists a bounded Palais-Smale
sequence $\{v_{n}\}\subset \mathcal{S}_{\Omega _{R},\rho }$ such that
\begin{equation}
\widetilde{J}_{R,s}(v_{n})\rightarrow \overline{C}_{R,s,\rho }\ \text{and}\
\widetilde{J}_{R,s}^{\prime }(v_{n})|_{_{\mathcal{S}_{\Omega _{R},\rho
}}}\rightarrow 0\ \text{as}\ n\rightarrow \infty .  \label{E29}
\end{equation}%
By (\ref{E29}), there exists a sequence $\{\lambda _{R,n}\}\subset \mathbb{R}
$ such that
\begin{equation}
\widetilde{J}_{R,s}^{\prime }(v_{n})+\lambda _{R,n}v_{n}=o_{n}(1)\text{ in }%
H_{0}^{-1}\left( \Omega _{R}\right) .  \label{E30}
\end{equation}%
From (\ref{E29})--(\ref{E30}), it follows that%
\begin{eqnarray*}
&&|\lambda _{R,n}|\rho \\
&=&\left\vert -\int_{\Omega _{R}}|\nabla v_{n}|^{2}dx-\alpha \int_{\Omega
_{R}}\int_{\Omega _{R}}\log \left( |x-y|\right)
|v_{n}(x)|^{2}|v_{n}(y)|^{2}dxdy+s\int_{\Omega _{R}}|v_{n}|^{p}dx\right\vert
\\
&\leq &\Vert \nabla v_{n}\Vert _{2,\Omega_{R}}^{2}-\alpha \rho ^{2}\log
(1+R)-\alpha \mathcal{CC}_{8/3}^{3/2}\rho ^{3/2}\Vert \nabla v_{n}\Vert
_{2,\Omega_{R}}+\mathcal{C}_{p}\rho \Vert \nabla v_{n}\Vert
_{2,\Omega_{R}}^{p-2},
\end{eqnarray*}%
which implies that $\{\lambda _{R,n}\}$ is bounded. Since $\{\lambda
_{R,n}\},\{v_{n}\}$ are bounded, there exists $(\lambda _{R,s},v_{R,s})\in
\mathbb{R}\times H_{0}^{1}\left( \Omega _{R}\right) $ such that up to a
subsequence,
\begin{equation*}
\lambda _{R,n}\rightarrow \lambda _{R,s}\text{ in }%
\mathbb{R}
,\text{ }v_{n}\rightharpoonup v_{R,s}\ \text{in}\ H_{0}^{1}\left( \Omega
_{R}\right) \ \text{and }v_{n}\rightarrow v_{R,s}\ \text{in}\ L^{q}\left(
\Omega _{R}\right) \ \text{for any}\ q\geq 1,
\end{equation*}%
and
\begin{equation}
\int_{\Omega _{R}}|\nabla v_{R,s}|^{2}dx+\alpha \int_{\Omega
_{R}}\int_{\Omega _{R}}\log \left( |x-y|\right)
|v_{R,s}(x)|^{2}|v_{R,s}(y)|^{2}dxdy-s\int_{\Omega
_{R}}|v_{R,s}|^{p}dx+\lambda _{R,s}\rho =0.  \label{E31}
\end{equation}%
Define $w_{n}=v_{n}-v_{R,s}$, then similar to the argument of Proposition %
\ref{L4.2}, we have%
\begin{equation*}
\Vert w_{n}\Vert _{p,\Omega _{R}}^{p}=o_{n}(1)\ \text{and}\ \int_{\Omega
_{R}}\int_{\Omega _{R}}\log \left( |x-y|\right)
|w_{n}(x)|^{2}|w_{n}(y)|^{2}dxdy=o_{n}(1).
\end{equation*}%
Using this and (\ref{E31}), we deduce that
\begin{eqnarray*}
o_{n}(1) &=&\int_{\Omega _{R}}|\nabla v_{n}|^{2}dx+\alpha \int_{\Omega
_{R}}\int_{\Omega _{R}}\log \left( |x-y|\right)
|v_{n}(x)|^{2}|v_{n}(y)|^{2}dxdy \\
&&-s\int_{\Omega _{R}}|v_{n}|^{p}dx+\lambda _{R,n}\int_{\Omega
_{R}}|v_{n}|^{2}dx \\
&=&\int_{\Omega _{R}}|\nabla v_{R,s}|^{2}dx+\int_{\Omega _{R}}|\nabla
w_{n}|^{2}dx+\alpha \int_{\Omega _{R}}\int_{\mathbb{R}^{2}}\log \left(
|x-y|\right) |v_{R,s}(x)|^{2}|v_{R,s}(y)|^{2}dxdy \\
&&-s\int_{\mathbb{R}^{2}}|v_{R,s}|^{p}dx+\lambda _{R,s}\rho +o_{n}(1) \\
&=&\int_{\Omega _{R}}|\nabla w_{n}|^{2}dx+o_{n}(1),
\end{eqnarray*}%
which indicates that $v_{n}\rightarrow v_{R,s}$ in $H_{0}^{1}\left( \Omega
_{R}\right) $. The proof is complete.
\end{proof}

\begin{lemma}
\label{L4.7} Assume that $\Omega $ is a star-shaped bounded domain. For any $%
\rho >0$ fixed, the set of solutions of Problem (\ref{E28}) is bounded
uniformly on $s$ and $R$.
\end{lemma}

\begin{proof}
Let $u$ be a solution of Problem (\ref{E28}). Then for $0<|\alpha |<\alpha
_{R,\rho }^{0}$ and $R>R_{\rho }$, by (\ref{E25}) one has
\begin{eqnarray*}
\widetilde{J}_{R,s}(u) &=&\frac{p-4}{2(p-2)}\int_{\Omega _{R}}|\nabla
u|^{2}dx+\frac{\alpha \rho ^{2}}{4(p-2)}+\frac{1}{2(p-2)}\int_{\partial
\Omega _{R}}|\nabla u|^{2}x\cdot nd\sigma \\
&&+\frac{\alpha }{4}\int_{\Omega _{R}}\int_{\Omega _{R}}\log
(|x-y|)|u(x)|^{2}|u(y)|^{2}dxdy \\
&\geq &\frac{p-4}{2(p-2)}\int_{\Omega _{R}}|\nabla u|^{2}dx+\frac{\alpha
\rho ^{2}}{4(p-2)}+\frac{\alpha }{4}\int_{\Omega _{R}}\int_{\Omega _{R}}\log
(1+|x-y|)|u(x)|^{2}|u(y)|^{2}dxdy \\
&\geq &\frac{p-4}{2(p-2)}\int_{\Omega _{R}}|\nabla u|^{2}dx+\frac{\alpha
\rho ^{2}}{4(p-2)}+\frac{\alpha }{4}\log (1+R)\rho ^{2} \\
&\geq &\frac{p-4}{2(p-2)}\int_{\Omega _{R}}|\nabla u|^{2}dx+\frac{\alpha
\rho ^{2}}{4(p-2)}-\frac{1}{2}\left( \frac{p}{(p-2)\rho \mathcal{C}_{p}}%
\right) ^{\frac{2}{p-4}},
\end{eqnarray*}%
which shows that
\begin{equation*}
\int_{\Omega _{R}}|\nabla u|^{2}dx\leq \frac{2(p-2)}{p-4}\left[ \widetilde{J}%
_{R,s}(u)-\frac{\alpha \rho ^{2}}{4(p-2)}+\frac{1}{2}\left( \frac{p}{%
(p-2)\rho \mathcal{C}_{p}}\right) ^{\frac{2}{p-4}}\right] .
\end{equation*}%
According to this and the fact that $\widetilde{J}_{R,s}(u)$ is bounded by
Lemma \ref{L4.4}, we complete the proof.
\end{proof}

\begin{theorem}
\label{L4.8} For any $R>R_{\rho }$ and $0<|\alpha |<\alpha _{R,\rho }^{\ast
} $, Problem (\ref{E5}) admits a mountain-pass solution $u_{R}^{1}\in
H_{0}^{1}\left( \Omega _{R}\right) $ for some Lagrange multiplier $\lambda
_{R}^{1}\in \mathbb{R}$.
\end{theorem}

\begin{proof}
By Proposition \ref{L4.6}, there exists a solution $(\lambda _{R,s},v_{R,s})$
to Problem (\ref{E28}) for almost every $s\in \left[ \frac{1}{2},1\right] $.
Furthermore, $v_{R,s}$ is bounded uniformly on $s$ by Lemma \ref{L4.7}.
Similar to the argument of Proposition \ref{L4.6}, there exist $u_{R}^{1}\in
\mathcal{S}_{R,\rho }$ and $\lambda _{R}^{1}\in \mathbb{R}$ such that up to
a subsequence,
\begin{equation*}
\lambda _{R,s}\rightarrow \lambda _{R}^{1}\text{ in }%
\mathbb{R}
\text{ and }v_{R,s}\rightarrow u_{R}^{1}\ \text{in}\ H_{0}^{1}(\Omega _{R})\
\text{as}\ s\rightarrow 1.
\end{equation*}%
So $(\lambda _{R}^{1},u_{R}^{1})$ is a solution of problem (\ref{E5}). The
proof is complete.
\end{proof}

\textbf{Proof of Theorem \ref{T1.3} $(iii)$:} By Theorem
\ref{L4.8}, we directly conclude that Theorem \ref{T1.3} $(iii)$ holds.

\subsection{The asymptotic behavior of solutions}

For convenience, let $\Omega _{R}=B_{R}(0)$. We study the asymptotic
behavior as $R\rightarrow \infty .$ Consider the following problem
\begin{equation}
\left\{
\begin{array}{ll}
-\Delta u+\lambda u=|u|^{p-2}u, & \ \text{in}\ \mathbb{R}^{2} \\
\int_{\mathbb{R}^{2}}|u|^{2}dx=\rho >0. &
\end{array}%
\right.  \label{E32}
\end{equation}%
Define
\begin{equation*}
m_{\rho }=\inf_{u\in \mathcal{P}_{\rho ,0}}J_{0}(u),
\end{equation*}%
where $J_{0}(u):=J(u)|_{\alpha =0}$ and $\mathcal{P}_{\rho ,0}:=\mathcal{P}%
_{\rho }|_{\alpha =0}$.

\begin{lemma}
\label{L4.9} (\cite[Lemma 2.1]{TZ2022}) Let $s,\rho>0$. We have

\begin{itemize}
\item[$(i)$] $v\in \mathcal{P}_{\rho ,0}$ if and only if $u(\cdot )=s^{-%
\frac{1}{p-4}}v\left( s^{-\frac{p-2}{2(p-4)}}\cdot \right) \in \mathcal{P}%
_{s\rho ,0}$ and $\mathcal{P}_{\rho ,0}\neq \emptyset $. Moreover, it holds
that $m_{s\rho }=s^{-\frac{2}{p-4}}m_{\rho }.$

\item[$(ii)$] If $p>4$, then $m_{\rho }>0$ is attained.
\end{itemize}
\end{lemma}

\begin{lemma}
\label{L4.10} Let $(\overline{\lambda }_{\rho },\overline{u}_{\rho })$ be
the solution of Problem (\ref{E32}). Then we have $\Vert \nabla \overline{u}%
_{\rho }\Vert _{2}\geq x_{\ast}:=\left[ \frac{p}{\rho \mathcal{C}_{p}(p-2)}%
\right] ^{\frac{2}{p-4}}$ and $m_{\rho }\geq \frac{p-4}{2(p-2)}\left( \frac{p%
}{(p-2)\rho \mathcal{C}_{p}}\right) ^{\frac{2}{p-4}}$.
\end{lemma}

\begin{proof}
Since $(\overline{\lambda }_{\rho },\overline{u}_{\rho })$ is the solution
of Problem (\ref{E32}), it is easy to obtain that
\begin{equation*}
m_{\rho }=\frac{1}{2}\int_{\mathbb{R}^{2}}|\nabla \overline{u}_{\rho
}|^{2}dx-\frac{1}{p}\int_{\mathbb{R}^{2}}|\overline{u}_{\rho }|^{p}dx
\end{equation*}%
and%
\begin{equation*}
P(u)|_{\alpha =0}=\int_{\mathbb{R}^{2}}|\nabla \overline{u}_{\rho }|^{2}dx-%
\frac{p-2}{p}\int_{\mathbb{R}^{2}}|\overline{u}_{\rho }|^{p}dx=0.
\end{equation*}%
These imply that $\int_{\mathbb{R}^{2}}|\nabla \overline{u}_{\rho
}|^{2}dx\geq \left[ \frac{p}{\rho \mathcal{C}_{p}(p-2)}\right] ^{\frac{2}{p-4%
}}$ and
\begin{equation*}
m_{\rho }=\frac{p-4}{2(p-2)}\int_{\mathbb{R}^{2}}|\nabla \overline{u}_{\rho
}|^{2}dx\geq \frac{p-4}{2(p-2)}\left( \frac{p}{(p-2)\rho \mathcal{C}_{p}}%
\right) ^{\frac{2}{p-4}}.
\end{equation*}%
The proof is complete.
\end{proof}

\begin{lemma}
\label{L4.11} Let $(\overline{\lambda }_{\rho },\overline{u}_{\rho })$ be
the solution of problem (\ref{E32}). Then there exists $R_{0}>0$ such that
\begin{equation*}
\overline{u}_{\rho }\leq C_{1}\exp ^{-C_{2}|x|}\ \text{for any}\ |x|\geq
R_{0},
\end{equation*}%
where $C_{1},C_{2}>0$.
\end{lemma}

\begin{proof}
Since $(\overline{\lambda }_{\rho },\overline{u}_{\rho })$ is the solution
of Problem (\ref{E32}), it follows from \cite[Lemma 2.1]{ZZ2022} that
\begin{equation*}
\overline{\lambda }_{\rho }=\rho ^{-\frac{p-2}{p-4}}\Vert W\Vert _{2}^{-%
\frac{2(p-2)}{p-4}}\ \text{and }\overline{u}_{\rho }=\overline{\lambda }%
_{\rho }^{\frac{1}{p-2}}W(\overline{\lambda }_{\rho }^{1/2}x),
\end{equation*}%
where $W$ is a positive solution to the equation%
\begin{equation*}
\left\{
\begin{array}{ll}
-\Delta W+W=|W|^{p-2}W, & \text{in }\mathbb{R}^{2}, \\
\lim_{x\rightarrow \infty }W(x)=0, &
\end{array}%
\right.
\end{equation*}%
which is unique up to translations in $\mathbb{R}^{2}$. Then we have
\begin{equation}
\overline{\lambda }_{\rho }>0\ \text{and}\ \overline{u}_{\rho }>0.
\label{E33}
\end{equation}%
Moreover,
\begin{equation}
\begin{array}{ll}
-\Delta \overline{u}_{\rho }=-\overline{\lambda }_{\rho }\overline{u}_{\rho
}+|\overline{u}_{\rho }|^{p-2}\overline{u}_{\rho }, & \ \text{in}\ \mathbb{R}%
^{2}.%
\end{array}
\label{E34}
\end{equation}%
Then by (\ref{E33})--(\ref{E34}), we obtain that $-\Delta \overline{u}_{\rho
}\in L_{loc}^{q}(\mathbb{R}^{2})$ for any $q>1$. Using \cite[Theorem 9.11]%
{GT1983}, we conclude that $\overline{u}_{\rho }\in W_{loc}^{2,q}(\mathbb{R}%
^{2})$ for all $q>1$, whence $\overline{u}_{\rho }\in C_{loc}^{1,r}(\mathbb{R%
}^{2})$ by the Sobolev embedding theorem for any $0<r<1$. So $\overline{u}%
_{\rho }\in L^{\infty }(\mathbb{R}^{2})\cup C_{loc}^{1,r}(\mathbb{R}^{2})$ .
Following (\ref{E34}) and \cite[Theorem 8.17]{GT1983}, there exists $C_{1}>0$
independent of $R$ such that
\begin{equation*}
\sup_{x\in B_{1/2}(y)}\overline{u}_{\rho }\leq C_{1}\Vert \overline{u}_{\rho
}\Vert _{2,B_{1}(y)},\ \text{for any}\ y\in \mathbb{R}^{2}.
\end{equation*}%
So there exists $R_{0}>0$ such that $\overline{u}_{\rho }^{p-2}\leq \frac{%
\overline{\lambda }_{\rho }}{2}$ for any $|x|\geq R_{0}$. Combining with (%
\ref{E34}), we have
\begin{equation}
-\Delta \overline{u}_{\rho }\leq -\frac{\overline{\lambda }_{\rho }}{2}%
\overline{u}_{\rho },\ \text{for any}\ |x|\geq R_{0}.  \label{E35}
\end{equation}%
Let $\ell =\sqrt{\frac{\overline{\lambda }_{\rho }}{2}}$ and $\eta (x)=\ell
^{2}\exp ^{\ell R_{0}}\exp ^{-\ell |x|}$ for any $|x|\geq R_{0}$. Then
\begin{equation}
\eta (x)=\overline{u}_{\rho }\ \text{for}\ \text{any }|x|=R_{0}\ \text{and}\
\Delta \eta (x)\leq \ell ^{2}\eta (x)\ \text{for any}\ |x|\geq R_{0}.
\label{E36}
\end{equation}%
Define $\widetilde{\eta }(x)=\eta (x)-\overline{u}_{\rho }$. By (\ref{E35}%
)--(\ref{E36}), one has
\begin{equation*}
\left\{
\begin{array}{l}
-\Delta \widetilde{\eta }(x)+\ell ^{2}\widetilde{\eta }(x)\geq 0,\ \text{for
all}\ |x|>R_{0} \\
\eta (x)=\overline{u}_{\rho },\ \text{for}\ |x|=R_{0}.%
\end{array}%
\right.
\end{equation*}%
The maximum principle (see \cite[Theorem 8.1]{GT1983}) implies that $%
\widetilde{\eta }(x)\geq 0$ for all $|x|>R_{0}$, and hence the conclusion
follows.
\end{proof}

By (\ref{E15}), (\ref{E26}) and Lemma \ref{L4.4}, we have
\begin{eqnarray}
\overline{C}_{R,\rho }:=\overline{C}_{R,s,\rho }|_{s=1} &\geq &f\left(
x_{\ast }\right)  \notag \\
&=&\frac{p-4}{2(p-2)}\left( \frac{p}{(p-2)\rho \mathcal{C}_{p}}\right) ^{%
\frac{2}{p-4}}+\frac{\alpha }{4}\rho ^{2}\log (1+R)  \notag \\
&\rightarrow &\frac{p-4}{2(p-2)}\left( \frac{p}{(p-2)\rho \mathcal{C}_{p}}%
\right) ^{\frac{2}{p-4}}\text{ as }R\rightarrow \infty .  \label{E37}
\end{eqnarray}

\begin{lemma}
\label{L4.12} Let $(\overline{\lambda }_{\rho },\overline{u}_{\rho })$ be
the solution of Problem (\ref{E32}). For $R>R_{\rho }$ large enough and $%
0<|\alpha |<\alpha _{R,\rho }^{\ast }$, there exists $\rho _{R}<\rho $
satisfying $\rho _{R}=\rho +o_{R}(1)$, such that
\begin{equation*}
\overline{C}_{R,\rho _{R}}\leq m_{\rho }+o_{R}(1).
\end{equation*}
\end{lemma}

\begin{proof}
Let $(\overline{\lambda }_{\rho },\overline{u}_{\rho })$ be the solution of
Problem (\ref{E32}). Set%
\begin{equation*}
H(t)=J_{0}\left( \left( \overline{u}_{\rho }\right) _{t}\right) =\frac{t^{2}%
}{2}\int_{\mathbb{R}^{2}}|\nabla \overline{u}_{\rho }|^{2}dx-\frac{t^{p-2}}{p%
}\int_{\mathbb{R}^{2}}|\overline{u}_{\rho }|^{p}dx\text{ for }t>0.
\end{equation*}%
Clearly, $H(t)$ is increasing on $0<t<1$ and is decreasing on $t>1,$ and $%
H(1)=\max_{t>0}J_{0}\left( \left( \overline{u}_{\rho }\right) _{t}\right) $.
Then by the fact of $\Vert \nabla \overline{u}_{\rho }\Vert _{2}\geq x_{\ast}
$ by Lemma \ref{L4.10}, there exists $0<s_{0}<1$ such that
\begin{equation}
\Vert \nabla \left( \overline{u}_{\rho }\right) _{s_{0}}\Vert _{2}<%
\left(\frac{p-4}{2(p-2)}\right)^{1/2} x_{\ast}\ \text{and}\
J_{0}\left( \left( \overline{u}_{\rho }\right) _{s_{0}}\right) <\frac{1}{2}%
m_{\rho }.  \label{E38}
\end{equation}

For any $R>\max \left\{ \frac{2R_{0}}{s_{0}},R_{\rho }\right\} $, we define $%
\xi _{R}\in C^{1}(\mathbb{R}^{2},[0,1])$ such that $|\nabla \xi _{R}|\leq 2$
and
\begin{equation*}
\xi _{R}(x)=\left\{
\begin{array}{l}
1,\ \text{if}\ x\in B_{Rs_{0}/2}(0), \\
0,\ \text{if}\ x\in \mathbb{R}^{2}\setminus B_{Rs_{0}}(0).%
\end{array}%
\right.
\end{equation*}%
Let $v_{R}=\xi _{R}\overline{u}_{\rho }$ and $\Vert v_{R}\Vert _{2}^{2}=\rho
_{R}$. Clearly, $\rho _{R}<\rho $. Since
\begin{equation*}
\int_{\mathbb{R}^{2}\setminus B_{Rs_{0}/2}(0)}|C_{1}e^{-C_{2}|x|}|^{2}dx\leq
C\int_{Rs_{0}/2}^{\infty }\exp ^{-2C_{2}r}rdr=o_{R}(1),
\end{equation*}%
It follows from Lemma \ref{L4.11} that
\begin{eqnarray}
\int_{\mathbb{R}^{2}}|v_{R}|^{2}dx &=&\int_{\mathbb{R}^{2}}|\overline{u}%
_{\rho }|^{2}dx+o_{R}(1),  \label{E39} \\
\int_{\mathbb{R}^{2}}|\nabla v_{R}|^{2}dx &=&\int_{\mathbb{R}^{2}}|\nabla
\overline{u}_{\rho }|^{2}dx+o_{R}(1),  \label{E40} \\
\int_{\mathbb{R}^{2}}|v_{R}|^{p}dx &=&\int_{\mathbb{R}^{2}}|\overline{u}%
_{\rho }|^{p}dx+o_{R}(1).  \label{E41}
\end{eqnarray}%
Note that $J_{0}\left( (v_{R})_{t}\right) $ has a unique maximum point $%
t_{R} $, i.e., $\max_{t>0}J_{0}\left( (v_{R})_{t}\right) =J_{0}\left(
(v_{R})_{t_{R}}\right) $. Then by (\ref{E39})--(\ref{E41}), we have
\begin{equation}
\max_{t>0}J_{0}\left( (v_{R})_{t}\right) =\max_{t>0}J_{0}\left( (\overline{u}%
_{\rho })_{t}\right) +o_{R}(1)=m_{\rho }+o_{R}(1),  \label{E42}
\end{equation}%
which shows that $t_{R}=1+o_{R}(1)$. Let $R$ large enough such that $%
t_{R}>s_{0}$. Then using (\ref{E42}), one has%
\begin{equation}
\max_{t\geq s_{0}}J_{0}\left( (v_{R})_{t}\right) =m_{\rho }+o_{R}(1).
\label{E43}
\end{equation}%
Note that $(v_{R})_{t}\in \mathcal{S}_{B_{R}(0),\rho_{R}}$ for any $t\geq
s_{0}$. Moreover, there exist $\widetilde{t}_{R}>\overline{t}_{R}>s_{0}$
such that
\begin{equation*}
\Vert \nabla \left( v_{R}\right) _{\widetilde{t}_{R}}\Vert _{2}>\Vert \nabla
\left( v_{R}\right) _{\overline{t}_{R}}\Vert _{2}=x_{\ast }\text{ and }%
\widetilde{J}_{R}\left( \left( v_{R}\right) _{\widetilde{t}_{R}}\right) <0.
\end{equation*}%
According to (\ref{E26}), (\ref{E39})--(\ref{E41}) and (\ref{E43}), we get
\begin{equation}
\max_{t\geq s_{0}}\widetilde{J}_{R}\left( (v_{R})_{t}\right) =\max_{t\geq
s_{0}}J_{0}\left( (v_{R})_{t}\right) +o_{R}(1)=m_{\rho }+o_{R}(1).
\label{E44}
\end{equation}%
Denote
\begin{equation*}
\chi(t)=\begin{cases}
0,\ t\geq s_{0},\\
1,\ 0\leq t<s_{0}.
\end{cases}
\end{equation*}
and
\begin{equation*}
\gamma (t)=
\chi(t)\left(w_{R,\rho _{R}}\right)_{\widetilde{C}_{R}t+1}+\left( 1-\chi(t)\right) \left(v_{R}\right)_{t},
\end{equation*}where $\widetilde{C}_{R}>0$ satisfies $\Vert\nabla \left(w_{R,\rho _{R}}\right)_{\widetilde{C}_{R}s_{0}+1}\Vert_{2,B_{R}(0)}=\Vert \nabla \left(v_{R}\right)_{s_{0}}\Vert_{2,B_{R}(0)}$, $w_{R,\rho _{R}}$ is the solution of Problem (\ref{E5}) with $\Omega
_{R}=B_{R}(0)$ satisfying $C_{R,\rho _{R}}=\widetilde{J}_{R}\left( w_{R,\rho
_{R}}\right) $. Then $\gamma (t)\in \Gamma _{R,\rho _{R}}$ by Lemma \ref%
{L4.4} and so
\begin{equation}
\overline{C}_{R,\rho _{R}}\leq \max_{t>0}\widetilde{J}_{R}\left( \gamma
(t)\right) .  \label{E45}
\end{equation}%
Note that $\Vert\nabla w_{R,\rho _{R}}\Vert_{2,B_{R}(0)}\geq R^{-1}\left[\lambda(B_{1}(0))\rho_{R}\right]^{1/2}$ by the Poincar\'{e} inequality. Following this, (\ref{E38})--(\ref{E40}), we deduce that
\begin{equation*}
\left(\frac{p-4}{2(p-2)}\right)^{1/2}
x_{\ast}+o_{R}(1)>\Vert\nabla \left(w_{R,\rho _{R}}\right)_{\widetilde{C}_{R}s_{0}+1}\Vert_{2,B_{R}(0)}
\geq \left(\widetilde{C}_{R}s_{0}+1\right)R^{-1}\left[\lambda(B_{1}(0))\rho\right]^{1/2}+o_{R}(1),
\end{equation*}
which implies that $\widetilde{C}_{R}<CRs_{0}^{-1}$, where $C=\left(\frac{p-4}{2(p-2)}\right)^{1/2}\left[\lambda(B_{1}(0))\rho\right]^{-1/2}x_{\ast}$.
So using (\ref{E26}), we have
\begin{equation}
\vert \alpha\vert \log\left(\widetilde{C}_{R}s_{0}+1\right)<\vert \alpha\vert \log\left(CR+1\right)<\vert \alpha\vert\left(\left|\log C\right|+\log(1+R)\right)=o_{R}(1).\label{E46}
\end{equation}
It follows from (\ref{E38}), (\ref{E40}) and (\ref{E46}) that
\begin{eqnarray}
&&\max_{0<t<s_{0}}\widetilde{J}_{R}\left(\left(w_{R,\rho _{R}}\right)_{\widetilde{C}_{R}t+1}\right) \notag\\
&=&\max_{0<t<s_{0}}\left[\frac{1}{2}\left(\widetilde{C}_{R}t+1\right)^{2}\int_{B_{R}(0)}|\nabla w_{R,\rho _{R}}|^{2}dx-\frac{\left(\widetilde{C}_{R}t+1\right)^{p-2}}{p}\int_{B_{R}(0)}|w_{R,\rho _{R}}|^{p}dx \right.\notag\\
&&\left.-\frac{\alpha}{4}\rho_{R}^{2}\log\left(\widetilde{C}_{R}t+1\right)+\frac{\alpha}{4}\int_{B_{R}(0)}\int_{B_{R}(0)}\log(\vert x-y\vert)|w_{R,\rho _{R}}(x)|^{2}|w_{R,\rho _{R}}(y)|^{2}dxdy\right] \notag\\
&=&\max_{0<t<s_{0}}\left[\frac{1}{2}\left(\widetilde{C}_{R}t+1\right)^{2}\int_{B_{R}(0)}|\nabla w_{R,\rho _{R}}|^{2}dx-\frac{\left(\widetilde{C}_{R}t+1\right)^{p-2}}{p}\int_{B_{R}(0)}|w_{R,\rho _{R}}|^{p}dx\right]+o_{R}(1) \notag\\
&\leq&\frac{1}{2}\Vert \nabla \left(v_{R}\right)_{s_{0}}\Vert_{2,B_{R}(0)}^{2}+o_{R}(1) \notag\\
&<&\frac{p-4}{4(p-2)}\left(\frac{p}{(p-2)\rho\mathcal{C}_{p}}\right)^{\frac{2}{p-4}}+o_{R}(1).\label{E47}
\end{eqnarray}
By (\ref{E44}) and Lemma \ref{L4.10}, we obtain
\begin{eqnarray}
\max_{t\geq s_{0}}\widetilde{J}_{R}\left( (v_{R})_{t}\right)\geq\frac{p-4}{2(p-2)}\left(\frac{p}{(p-2)\rho\mathcal{C}_{p}}\right)^{\frac{2}{p-4}}+o_{R}(1),\label{E48}
\end{eqnarray}
Combining with (\ref{E47}) and (\ref{E48}), we conclude that
$\max_{t>0}\widetilde{J}_{R}\left( \gamma
(t)\right) =\max_{t>s_{0}}\widetilde{J}_{R}\left( \left(v_{R}\right)_{t}\right)$.
Using this, together with (\ref{E44})--(\ref{E45}), for $R$ large enough,
we have
\begin{eqnarray*}
\overline{C}_{R,\rho _{R}} \leq \max_{t>0}\widetilde{J}_{R}\left( \gamma
(t)\right) =\max_{t>s_{0}}\widetilde{J}_{R}\left(
\left(v_{R}\right)_{t}\right) =m_{\rho }+o_{R}(1).
\end{eqnarray*}%
The proof is complete.
\end{proof}

\textbf{Proof of Theorem \ref{T1.5}:} Let $\left( \lambda
_{R}^{0},u_{R}^{0}\right) \in \mathbb{R}\times H_{0}^{1}(B_{R}(0))$ be the
local minimizer of Problem (\ref{E5}). Then
\begin{equation}
C_{R,\rho }=\frac{1}{2}\int_{B_{R}(0)}|\nabla u_{R}^{0}|^{2}dx+\frac{\alpha
}{4}\int_{B_{R}(0)}\int_{B_{R}(0)}\log
(|x-y|)|u_{R}^{0}(x)|^{2}|u_{R}^{0}(y)|^{2}dxdy-\frac{1}{p}%
\int_{B_{R}(0)}|u_{R}^{0}|^{p}dx  \label{E49}
\end{equation}%
and
\begin{equation}
\int_{B_{R}(0)}|\nabla u_{R}^{0}|^{2}dx-\frac{p-2}{p}%
\int_{B_{R}(0)}|u_{R}^{0}|^{p}dx-\frac{\alpha }{4}\rho ^{2}=\frac{R}{2}%
\int_{\partial B_{R}(0)}|\nabla u_{R}^{0}|^{2}dx\geq 0.  \label{E50}
\end{equation}%
For $0<|\alpha |<\alpha _{R,\rho }^{\ast },$ it follows from (\ref{E26})--(%
\ref{E27}) and (\ref{E49})--(\ref{E50}) that
\begin{eqnarray*}
\int_{B_{R}(0)}|\nabla u_{R}^{0}|^{2}dx &\leq &\frac{2(p-2)}{p-4}\left(
C_{R,\rho }-\frac{\alpha \rho ^{2}}{4(p-2)}-\frac{\alpha }{4}\log (1+R)\rho
^{2}\right) \\
&\leq &\frac{2(p-2)}{p-4}\left( C_{R,\rho }-\frac{\alpha \rho ^{2}}{4(p-2)}-%
\frac{\alpha }{4}\log (1+R)\rho ^{2}\right) \\
&\rightarrow &0\ \text{as}\ R\rightarrow \infty .
\end{eqnarray*}

Assume that the sequences $\{R_{n}\},\{\rho _{n}\}\subset
\mathbb{R}
^{+},\{\alpha _{n}\}\subset
\mathbb{R}
^{-}$ satisfy $\lim_{n\rightarrow \infty }R_{n}=\infty ,$ $0<\rho
_{n}<\lim_{n\rightarrow \infty }\rho _{n}=\rho $ given, and $0<|\alpha
_{n}|<\alpha _{R_{n},\rho _{n}}^{\ast },$ respectively. Let $(\lambda
_{R_{n}}^{1},u_{R_{n}}^{1})\in \mathbb{R}\times H_{0}^{1}(B_{R_{n}}(0))$ be
the solutions of the following problem%
\begin{equation*}
\left\{
\begin{array}{ll}
-\Delta u+\lambda u+\alpha _{n}\left( \log |\cdot |\ast |u|^{2}\right)
u=|u|^{p-2}u, & \ \text{in}\ B_{R_{n}}(0), \\
\int_{B_{R_{n}}(0)}|u|^{2}dx=\rho _{n}>0. &
\end{array}%
\right.
\end{equation*}%
Moreover, $\widetilde{J}_{R_{n}}(u_{R_{n}}^{1})=\overline{C}_{R_{n},\rho
_{n}}$. By Lemma \ref{L4.7}, the sequences $\{\lambda
_{R_{n}}^{1}\},\{u_{R_{n}}^{1}\}$ are bounded. Then there exist $\lambda
_{0}\in \mathbb{R}$ and $v_{0}\in H^{1}(\mathbb{R}^{2})\setminus \{0\}$ such
that up to a subsequence,
\begin{equation*}
\lambda _{R_{n}}^{1}\rightarrow \lambda _{0}\ \text{in}\ \mathbb{R}\ \text{%
and}\ u_{R_{n}}^{1}\rightharpoonup v_{0}\ \text{in}\ H^{1}(\mathbb{R}^{2}),
\end{equation*}%
and $(\lambda _{0},v_{0})\in \mathbb{R}\times H^{1}(\mathbb{R}^{2})$ is the
solution to $-\Delta u+\lambda u=|u|^{p-2}u$ in $\mathbb{R}^{2}$. Similar to
(\ref{E26}),
\begin{equation}
\alpha _{n}\log (1+R_{n})=o_{n}(1).  \label{E51}
\end{equation}%
By (\ref{E24}) and (\ref{E51}), we have
\begin{eqnarray}
\lim_{n\rightarrow \infty }\lambda _{R_{n}}^{1}\rho _{n}
&=&\lim_{n\rightarrow \infty }\left( -\frac{\alpha _{n}}{4}\left(
\int_{B_{R_{n}}(0)}|u_{R_{n}}^{1}|^{2}dx\right) ^{2}+\frac{2}{p}%
\int_{B_{R_{n}}(0)}|u_{R_{n}}^{1}|^{p}dx-\frac{1}{2}\int_{\partial
B_{R_{n}}(0)}|\nabla u_{R_{n}}^{1}|^{2}x\cdot nd\sigma \right.  \notag \\
&&\left. -\alpha _{n}\int_{B_{R_{n}}(0)}\int_{B_{R_{n}}(0)}\log
(|x-y|)|u_{R_{n}}^{1}(x)|^{2}|u_{R_{n}}^{1}(y)|^{2}dxdy\right)  \notag \\
&=&\frac{2}{p}\lim_{n\rightarrow \infty
}\int_{B_{R_{n}}(0)}|u_{R_{n}}^{1}|^{p}dx.  \label{E52}
\end{eqnarray}%
We now prove that $\int_{B_{R_{n}}(0)}|u_{R_{n}}^{1}|^{p}dx\neq o_{n}(1)$.
Otherwise, by (\ref{E52}), we have $\lim_{n\rightarrow \infty }\lambda
_{R_{n}}^{1}\rho _{n}=0$ and thus
\begin{eqnarray*}
\int_{B_{R_{n}}(0)}|\nabla u_{R_{n}}^{1}|^{2}dx &=&-\lambda _{R_{n}}^{1}\rho
_{n}+\int_{B_{R_{n}}(0)}|u_{R_{n}}^{1}|^{p}dx \\
&&-\alpha _{n}\int_{B_{R_{n}}(0)}\int_{B_{R_{n}}(0)}\log
(|x-y|)|u_{R_{n}}^{1}(x)|^{2}|u_{R_{n}}^{1}(y)|^{2}dxdy \\
&=&o_{n}(1),
\end{eqnarray*}%
which contradicts with (\ref{E37}). So $\lambda _{0}>0$.

Let $v_{n}^{1}=u_{R_{n}}^{1}-v_{0}$. Since $-\Delta v_{0}+\lambda
_{0}v_{0}=|v_{0}|^{p-2}v_{0}$, it holds $\int_{\mathbb{R}^{2}}|\nabla
v_{0}|^{2}dx-\frac{p-2}{p}\int_{\mathbb{R}^{2}}|v_{0}|^{p}dx=0$. Using this,
together with (\ref{E25}), (\ref{E51}) and the Brezis-Lieb lemma, yields
\begin{eqnarray*}
&&\int_{\mathbb{R}^{2}}|\nabla v_{n}^{1}|^{2}dx+\int_{\mathbb{R}^{2}}|\nabla
v_{0}|^{2}dx+o_{n}(1) \\
&=&\int_{B_{R_{n}}(0)}|\nabla u_{R_{n}}^{1}|^{2}dx \\
&=&\frac{p-2}{p}\int_{B_{R_{n}}(0)}|u_{R_{n}}^{1}|^{p}dx+\frac{\alpha
_{n}\rho ^{2}}{4}+\frac{R_{n}}{2}\int_{\partial B_{R_{n}}(0)}|\nabla
u_{R_{n}}^{1}|^{2}d\sigma \\
&=&\frac{p-2}{p}\int_{\mathbb{R}^{2}}|u_{R_{n}}^{1}|^{p}dx+o_{n}(1) \\
&=&\frac{p-2}{p}\int_{\mathbb{R}^{2}}|u_{R_{n}}^{1}|^{p}dx+\frac{p-2}{p}%
\int_{\mathbb{R}^{2}}|v_{0}|^{p}dx+o_{n}(1) \\
&=&\frac{p-2}{p}\int_{\mathbb{R}^{2}}|v_{n}^{1}|^{p}dx+\int_{\mathbb{R}%
^{2}}|\nabla v_{0}|^{2}dx+o_{n}(1),
\end{eqnarray*}%
which shows that $\int_{\mathbb{R}^{2}}|\nabla v_{n}^{1}|^{2}dx=\frac{p-2}{p}%
\int_{\mathbb{R}^{2}}|v_{n}^{1}|^{p}dx+o_{n}(1)$ and so $J_{0}(v_{n}^{1})=%
\frac{p-4}{2(p-2)}\int_{\mathbb{R}^{2}}|\nabla v_{n}^{1}|^{2}dx+o_{n}(1)$.
If $\int_{\mathbb{R}^{2}}|\nabla v_{n}^{1}|^{2}dx=o_{n}(1)$, then $%
u_{R_{n}}^{1}\rightarrow v_{0}$ in $H^{1}(\mathbb{R}^{2})$. Otherwise, there
holds $J_{0}(v_{n}^{1})\geq \frac{p-4}{2(p-2)}\left( \frac{p}{(p-2)\rho
\mathcal{C}_{p}}\right) ^{\frac{2}{p-4}}+o_{n}(1)$. Then by Lemmas \ref{L4.9}
$(i)$ and \ref{L4.12}, one has%
\begin{eqnarray*}
m_{\rho }+o_{n}(0)\geq \overline{C}_{R_{n},\rho _{n}}=\widetilde{J}%
_{R_{n}}(u_{R_{n}}^{1}) &=&J_{0}(v_{n}^{1})+J_{0}(v_{0})+o_{n}(1) \\
&\geq &m_{\rho }+\frac{p-4}{2(p-2)}\left( \frac{p}{(p-2)\rho \mathcal{C}_{p}}%
\right) ^{\frac{2}{p-4}}+o_{n}(1),
\end{eqnarray*}%
which is impossible. Therefore, $(\lambda _{0},v_{0})\in
\mathbb{R}
^{+}\times H^{1}(\mathbb{R}^{2})$ is the solution of Problem (\ref{E32}).

\section*{Acknowledgments}

J. Sun was supported by National Natural Science Foundation of China (Grant
No. 12371174). S. Yao was supported by National Natural Science Foundation
of China (Grant No. 12501218) and Natural Science Foundation of Shandong
Province (Grant No. ZR2025QC1512).

\end{document}